\documentclass[a4paper,oneside,11pt,reqno]{amsart}
\usepackage{amsthm}
\usepackage{amsfonts}
\usepackage{amsmath}
\usepackage{amssymb}
\usepackage{mathrsfs}
\usepackage{epsfig}
\usepackage{pst-plot}
\usepackage{ifthen}
\usepackage[active]{srcltx}
\usepackage{tikz}

\reversemarginpar  

\numberwithin{equation}{section}

\newtheorem{theorem}{\bf Theorem}[section]
\newtheorem{proposition}[theorem]{\bf Proposition}
\newtheorem{corollary}[theorem]{\bf Corollary}
\newtheorem{lemma}[theorem]{\bf Lemma}

\theoremstyle{remark}
\newtheorem{definition}[theorem]{\bf Definition}
\newtheorem{example}[theorem]{\bf Example}
\newtheorem{question}[theorem]{\bf Question}
\newtheorem{remark}[theorem]{\bf Remark}

\newcommand*{\Ge}{\geqslant}
\newcommand*{\hh}{\mathcal{H}}
\newcommand*{\inp}[2]{\langle{#1},\,{#2} \rangle}
\newcommand*{\Le}{\leqslant}
\newcommand*{\ob}[1]{{\mathcal R}(#1)}

\newcommand*{\zbb}{\mathbb{Z}}
\def\D{\mathbb{D}}

\begin{document}
\title[Weakly concave operators]{Weakly concave operators}

   \author{Sameer Chavan}
   \address{Department of Mathematics and Statistics\\
Indian Institute of Technology Kanpur, India}
   \email{chavan@iitk.ac.in}

   \author[J.\ Stochel]{Jan Stochel}
   \address{Instytut Matematyki, Uniwersytet Jagiello\'nski,
ul.\ {\L}ojasiewicza 6, 30-348 Kra\-k\'ow,
Poland}
   \email{Jan.Stochel@im.uj.edu.pl}



   \begin{abstract}
We study a class of left-invertible operators
which we call weakly concave operators. It
includes the class of concave operators and some
subclasses of expansive strict $m$-isometries
with $m > 2$. We prove a Wold-type decomposition
for weakly concave operators. We also obtain a
Berger-Shaw-type theorem for analytic finitely
cyclic weakly concave operators. The proofs of
these results rely heavily on a spectral
dichotomy for left-invertible operators. It
provides a fairly close relationship, written in
terms of the reciprocal automorphism of the
Riemann sphere, between the spectra of a
left-invertible operator and any of its left
inverses. We further place the class of weakly
concave operators, as the term $\mathcal A_1$,
in the chain $\mathcal A_0 \subseteq \mathcal
A_1 \subseteq \ldots \subseteq \mathcal
A_{\infty}$ of collections of left-invertible
operators. We show that most of the
aforementioned results can be proved for members
of these classes. Subtleties arise depending on
whether the index $k$ of the class $\mathcal
A_k$ is finite or not. In particular, a
Berger-Shaw-type theorem fails to be true for
members of~$\mathcal A_{\infty}$. This
discrepancy is better revealed in the context of
$C^*$- and $W^*$-algebras.
   \end{abstract}

\maketitle


\section{Introduction}
Denote by $\zbb_+,$ $\mathbb Z$ and $\mathbb{R}$ the sets of
non-negative integers, integers and real numbers,
respectively. For a subset $\varOmega$ of the complex plane
$\mathbb C$, let $\partial \varOmega$, ${\overline
{\varOmega}}$ and ${\mathbb C} \setminus \varOmega$ stand for
the boundary, the closure and the complement of $\varOmega$
in $\mathbb C,$ respectively. For a complex number $z,$ the
conjugate and the absolute value of $z$ is denoted by
$\bar{z}$ and $|z|,$ respectively. Write $\varOmega^{*}$ for
the set $\{\bar{z} \colon z \in \varOmega\}$. For a positive
real number $r,$ $\mathbb D_r$ stands for the open disc $\{z
\in \mathbb C \colon |z| < r\}.$ Set $\mathbb D := \mathbb
D_1.$ For positive real numbers $r$ and $R$ with $r < R$,
denote by $\mathbb A(r, R)$ the annulus $\{z \in \mathbb C
\colon r < |z| < R\}$. Write $\mathbb{R}[x]$ for the ring of
polynomials in one real variable $x$ with real coefficients.

Unless otherwise stated, all Hilbert spaces appearing below
are complex. Let $\hh$ denote a Hilbert space and let
$\mathcal B(\hh)$ stand for the unital $C^*$-algebra of
bounded linear operators on $\hh$ with the identity operator
$I$ as the unit. We write $\ker T$ and $\ob{T}$ for the
kernel and the range of an operator $T \in \mathcal B(\hh).$
The symbol $|T|$ denotes the positive square root
$(T^*T)^{1/2}$ of $T^*T$. The {\it approximate point
spectrum} $\sigma_{ap}(T)$ of $T$ is the set of those
$\lambda$ in $\mathbb C$ for which $T-\lambda I$ is not
bounded below, and the {\it spectrum} $\sigma(T)$ is the
complement of the set of those $\lambda$ in $\mathbb C$ for
which ${(T-\lambda I)}^{-1}$ exists as a bounded linear
operator on $\hh$. The spectral radius $r(T)$ of $T$ is given
by $r(T)= \sup \{|\lambda| \colon \lambda \in \sigma(T)\}.$
It turns out that $\sigma(T)$ is a nonempty compact subset of
$\mathbb C$ (provided $\hh \neq \{0\}$), and hence $r(T)$ is
finite (see \cite[Theorem 2.2.9]{Si}). By the {\it inner
radius} $i(T)$ of $T$, we understand
   \begin{align} \label{iT}
i(T):= \lim_{n \to \infty}
\Big(\inf_{\|x\|=1}\|T^nx\|\Big)^{1/n}.
   \end{align}
It turns out that $i(T) \leqslant r(T)$, and hence $i(T)$
is finite. Moreover, by \cite[Proposition 13]{Sh}, we
have
   \begin{align*}
\sigma_{ap}(T) \subseteq \{\lambda \in \mathbb C
\colon i(T) \leqslant |\lambda| \leqslant
r(T)\}.
   \end{align*}

An operator $T$ in $\mathcal B(\hh)$ is {\it left-invertible}
if there exists $L \in \mathcal B(\hh)$ (a {\it left-inverse}
of $T$) such that $LT=I.$ Note that $T$ is left-invertible if
and only if $0\notin \sigma_{ap}(T)$, or equivalently if and
only if $T^*T$ is invertible (as a member of $\mathcal
B(\hh)$). Following \cite{S01}, we call the operator
$T':=T(T^*T)^{-1}$ the {\it Cauchy dual} of $T$ whenever $T$
is left-invertible. The Cauchy dual operator provides a
canonical choice of a left-inverse $L:=T'^*$ for any
left-invertible operator $T.$ We refer the reader to
\cite{S01} for the role that the Cauchy dual operator played
in the solution of the wandering subspace problem for the
Bergman shift. Straightforward computations yield the
following properties of the Cauchy dual, where in general
$P_{\mathcal M}$ denotes the orthogonal projection of $\hh$
onto a closed vector subspace $\mathcal M$ of $\hh$.
   \begin{proposition}  \label{polur}
If $T\in \mathcal B(\hh)$ is left-invertible and $T=U|T|$ is
the polar decomposition of $T$, then
   \begin{enumerate}
   \item[(i)]
$U=T|T|^{-1}$ and $|T'|=|T|^{-1},$
   \item[(ii)] $T'= U|T'|$  is the
polar decomposition of $T'$,
   \item[(iii)]
$P_{\ob{T}} = T'T^* = UU^*$.
   \end{enumerate}
   \end{proposition}
The {\it left spectrum} $\sigma_l(T)$ of $T \in
\mathcal B(\hh)$ is defined as the set of those
$\lambda$ in $\mathbb C$ for which $T-\lambda I$
is not left-invertible. Similarly, one can
define the {\it right-invertible} operator and
the {\it right spectrum} $\sigma_r(T)$ of $T.$
It is easy to see that
   \begin{align*}
\sigma_l(T)=\sigma_{ap}(T)=\{\lambda \in \mathbb
C \colon\overline{\lambda} \in \sigma_r(T^*)\}.
   \end{align*}
We say that $T$ is {\it Fredholm} if $\ob{T}$ is
closed and both $\dim \ker T$ and $\dim \ker
T^*$ are finite, where $\dim$ stands for the
Hilbert space dimension. The {\it essential
spectrum} $\sigma_e(T)$ of $T$ is the complement
of the set of those $\lambda \in \mathbb C$ for
which $T-\lambda I$ is Fredholm. The {\it
Fredholm index} $\mbox{ind}_T \colon \mathbb C
\setminus \sigma_e(T) \to \mathbb Z$ is given by
   \begin{align*}
\mbox{ind}_T(\lambda)\,:=\,\dim \ker(T-\lambda
I)-\dim \ker (T^*-\bar{\lambda} I), \quad
\lambda \in \mathbb C \setminus \sigma_e(T).
   \end{align*}

   An operator $T$ in ${\mathcal B}(\hh)$ is
{\it of trace class} if $\|T\|_{1, \mathcal E}:=
\sum_{n =0}^{\infty}\inp{|T|e_n}{e_n}$ is finite
for every choice of an orthonormal basis
$\mathcal E=\{e_n\}_{n \in \zbb_+}$ of $\hh.$
The number $\|T\|_{1, \mathcal E}$ turns out to
be independent of the choice of an orthonormal
basis $\mathcal E$ of $\hh$ and is denoted by
$\|T\|_1;$ it is called the {\it trace norm} of
$T.$ In case $T$ is of trace class, the {\it
trace} of $T$ given by $\mbox{tr}(T) = \sum_{n =
0}^{\infty}\inp{Te_n}{e_n}$ is a finite number.
Again, it is independent of the choice of an
orthonormal basis $\mathcal E$ of $\hh$ (see
\cite{Si} for more details on trace class
operators).

An operator $T$ in ${\mathcal B}(\hh)$ is said
to be {\it finitely cyclic} if there is a finite
dimensional vector subspace $\mathcal{M}$ of
$\mathcal{H}$, called a {\em cyclic space} of
$T$, such that
   \begin{align*}
\hh = \bigvee_{n \in \zbb_+} T^n (\mathcal{M}).
   \end{align*}
The following fact is well known (see
\cite[Proposition~1(i)]{H} and
\cite[Lemma~3.6]{ACJS}).
   \begin{align} \label{fundim}
   \begin{minipage}{60ex}
{\em If $T \in \mathcal B(\hh)$ is finitely
cyclic, then $\ker(T^*-\lambda I)$ is finite
dimensional for every $\lambda \in \mathbb C$.}
   \end{minipage}
   \end{align}
An operator $T \in \mathcal B(\hh)$ is {\it
analytic} if $\bigcap_{n=0}^{\infty}\ob{T^n}
 =\{0\}.$ Note that an analytic operator on a
non-zero Hilbert space is never invertible. We
say that $T \in \mathcal B(\hh)$ has the {\it
wandering subspace property} if
   \begin{align*}
\hh = \bigvee_{n \in \zbb_+} T^n(\ker T^*).
   \end{align*}
Recall that an operator $T\in \mathcal B(\hh)$
is said to be {\em completely non-unitary}
(resp., {\em completely non-normal}\/) if there
is no nonzero closed subspace of $\hh$ reducing
$T$ to a unitary (resp., normal) operator. The
following fact is easy to verify.
   \begin{align} \label{wsp-cnu}
   \begin{minipage}{65ex}
{\em If $T \in \mathcal B(\hh)$ has the
wandering subspace property, then $T$ is
completely non-unitary.}
   \end{minipage}
   \end{align}

   Denote by $[A,B]$ the {\em commutator} $AB-BA$ of $A,B \in
\mathcal{B}(\mathcal{H}).$ An operator $T \in \mathcal
B(\hh)$ is said to be {\it hyponormal} if the self-commutator
$[T^*, T]$ of $T$ is positive. We say that $T$ is {\it
cohyponormal} if $T^*$ is hyponormal. The reader is referred
to \cite[Chapter IV]{Co1} and \cite[Chapter 1]{M-P} for the
fundamentals of the theory of hyponormal operators including
the statements of Berger-Shaw Theorem and Putnam's
inequality. Given a positive integer $m$ and $T \in \mathcal
B(\hh)$, we set
   \begin{align*}
B_m(T) = \sum_{k=0}^m (-1)^k {m \choose k}{T^*}^kT^k.
   \end{align*}
  We say that $T$ is a {\it contraction} (resp.\ an {\it
expansion}) if $B_1(T) \geqslant 0$ (resp., \mbox{$B_1(T)
\leqslant 0$}). An operator $T$ is an {\it isometry} if
$B_1(T)=0$. For an integer $m \Ge 2$, an operator $T \in
\mathcal B(\hh)$ is said to be an {\it $m$-isometry} (resp.\
an {\it $m$-expansion}) if $B_m(T) = 0$ (resp.\ $B_m(T)
\leqslant 0$). By a {\it strict $m$-isometry}, we understand
an $m$-isometry which is not an $(m-1)$-isometry (or not an
isometry if $m=2$). A $2$-expansion is sometimes called a
{\it concave} operator. It is well known that any concave
operator is expansive (see \cite[Lemma 1]{R}). Furthermore,
the approximate point spectrum of any $m$-isometry and any
concave operator is contained in the unit circle $\partial
\mathbb D$ (see \cite[Lemma 1.21]{Ag-St}, \cite[Lemma
3.1]{JS} and \cite[Lemma 1]{R}). In particular, the Cauchy
dual operator of an $m$-isometry or a concave operator is
well defined.

The following theorem, which is one of the main results of
this paper, generalizes \cite[Theorem 3.6(A)]{S01} (for the
notion of weak concavity, see Definition~\ref{def}).
   \begin{theorem} \label{wsp}
Suppose that $T \in \mathcal B(\hh)$ is a weakly
concave operator. Let $\hh_u$ and $\hh'_u$
denote the closed vector subspaces of $\hh$
given by
   \begin{align} \label{Hu}
\text{$\hh_u := \bigcap_{n=0}^{\infty}\ob{T^n}$
\hspace{.2ex} and \hspace{.2ex} $\hh'_u :=
\bigcap_{n=0}^{\infty}\ob{T'^n}.$}
   \end{align}
Then $\hh'_u \subseteq \hh_u$ and $\hh'_u$
reduces $T.$ Moreover,
   \begin{align} \label{main-deco}
T = U \oplus S ~\mbox{with respect to the
decomposition} ~\hh = \hh'_u \oplus
(\hh'_u)^{\perp},
   \end{align}
where $U \in \mathcal B(\hh'_u)$ is unitary and
$S \in \mathcal B((\hh'_u)^{\perp})$ has the
wandering subspace property.
   \end{theorem}
In general, the Wold-type decomposition obtained in
Theorem~\ref{wsp} is weaker than Shimorin's one in the sense
that the operator $S$ in the decomposition \eqref{main-deco}
of $T$ may not be analytic. However, in the case where $T$ is
concave, it can be seen that $\hh'_u= \hh_u,$ which allows us
to recover \cite[Theorem 3.6(A)]{S01} (see
Corollary~\ref{thm-gen}). The proof of Theorem~\ref{wsp} is
based on the circle of ideas developed by Shimorin in
\cite{S01} together with the spectral dichotomy principle,
which is given in Section~\ref{Sec.4}. We also prove the
following Berger-Shaw-type theorem together with the
Carey-Pincus trace formula for analytic finitely cyclic
weakly concave operators (cf.\ \cite[Theorem~2.1, p.\
152]{Co1} and \cite[Proposition~2.21]{C07}).
   \begin{theorem}  \label{bst}
Suppose that $T \in \mathcal B(\hh)$ is an
analytic finitely cyclic weakly concave
operator. Then $T$ and $T'$ have trace class
self-commutators. Moreover, for any complex
polynomials $p, q$ in two complex variables $z$
and $\bar z$,
   \begin{align*}
\mathrm{tr}\,[p(T', T'^*), q(T', T'^*)] &=
\mathrm{tr}\,[p(T, T^*), q(T, T^*)] \\ &=
\frac{1}{\pi} \dim \ker T^* \int_{\D}
\Big(\frac{\partial p}{\partial \bar
z}\frac{\partial q}{\partial z} - \frac{\partial
p}{\partial z}\frac{\partial q}{\partial \bar
z}\Big) (z, \bar z) dA(z),
   \end{align*}
where $dA$ denotes the Lebesgue planar measure.
   \end{theorem}
The proofs of Theorems~\ref{wsp} and \ref{bst} will be given
in Section~\ref{Sec.5}.
   \begin{remark}
Let $T$ be as in Theorem~\ref{bst}. Then, in
particular, we have
   \begin{align*}
\mathrm{tr}\,[T'^*, T'] = \mathrm{tr}\,[T^*, T]
= \dim \ker T^*.
   \end{align*}
Further, the {\it Pincus principal function} of
$T'$ is equal to $-\dim \ker T^*$ times the
characteristic function of the open unit disk
(see \eqref{leba}). The reader is referred to
\cite{C-P77} for the definition of the Pincus
principal function.
   \hfill $\diamondsuit$
   \end{remark}
The paper is organized as follows. In Section~\ref{Sec.2}, we
formally introduce the class of weakly concave operators. In
particular, we discuss their elementary properties, including
their relation to the theory of hyponormal operators. In
Section~\ref{Sec.3}, we provide examples of $3$-isometric
weakly concave operators which are not concave (see
Example~\ref{przyk}). In Section~\ref{Sec.4}, we establish
the spectral dichotomy principle for left-invertible
operators (see Theorem~\ref{main}). This principle is a key
component of the proofs of Theorems~\ref{wsp} and \ref{bst}
(see Section~\ref{Sec.5}). Both theorems are generalized in
Section~\ref{Sec.6} to the case of operators of class
$\mathcal A_k$, where $k\in \zbb_+ \cup \{\infty\}$ (see
Theorems~\ref{Theorem2.7*} and \ref{bst-Jan}). If the
underlying Hilbert space is infinite dimensional, then the
class $\mathcal A_0$ is larger than the class of concave
operators, the class $\mathcal A_1$ consists precisely of
weakly concave operators and the sequence $\{\mathcal
A_k\}_{k=0}^{\infty}$ is strictly increasing (see
Example~\ref{separ-1}). In Section~\ref{Sec.7}, we present
necessary and sufficient conditions for some restrictions of
the Cauchy dual operator of a left-invertible operator to be
hyponormal (see Propositions~\ref{ajaj-1} and \ref{inft-1}).
The last section is devoted to $C^*$- and $W^*$-algebraic
analogues of the classes $\mathcal A_k.$
   \section{\label{Sec.2}An invitation to weakly concave operators}
It is well known that any concave operator
admits a Wold-type decomposition (see
\cite[Theorem 3.6]{S01}). Since concave
$m$-isometries are necessarily $2$-isometries
(see \cite[Proposition 4.9]{CC}), the above
result is insufficient to obtain a Wold-type
decomposition for concave strict $m$-isometries
if $m > 2.$ As a first step towards obtaining
the Wold-type decomposition for expansive
$m$-isometries, we consider a class of operators
containing concave operators and some strict
$m$-isometries with $m> 2.$ Furthermore, in view
of the fact that the Cauchy dual of an expansive
$m$-isometry is not necessarily hyponormal (in
contrast to the case of concave operators, see
\cite[p.\ 294]{Sh02} or
\cite[Theorem~2.9]{C07}), it is natural to
search for a class of left-invertible operators
satisfying a hyponormality-like condition. This
motivates us to define the following notion,
which is central to the present article.
   \begin{definition} \label{def}
A left-invertible operator $T \in \mathcal
B(\hh)$ is said to be {\it weakly concave} if it
satisfies the following conditions:
   \begin{align} \label{wc1}
\sigma_{ap}(T) & \subseteq \partial \mathbb D,
   \\ \label{wc2}
\|x\| & \leqslant \|Tx\|, \quad x \in \hh,
   \\ \label{wc}
\|T^*x\|' & \leqslant \|Tx\|', \quad x \in \hh,
   \end{align}
where $\|\cdot\|'$ is the norm on $\hh$ given by
$\|x\|':=\|T'x\|$ for $x \in \hh$ and $T'$
denotes the Cauchy dual of $T$.
   \end{definition}
   \begin{remark} \label{rmk-Jan}
In view of \cite[Lemma~3.1]{JS} and the inclusion $\partial
\sigma(T) \subseteq \sigma_{ap}(T)$ (see
\cite[Proposition~VII.6.7]{Co}), it is easy to see that,
under the assumption \eqref{wc2}, the condition \eqref{wc1}
is equivalent to $\sigma(T) \subseteq \overline{\mathbb D}$
(and to $r(T) \leqslant 1$).
   \hfill $\diamondsuit$
   \end{remark}
According to \cite[Lemma~1]{R} and Remark~\ref{rmk-Jan},
concave operators satisfy the conditions \eqref{wc1} and
\eqref{wc2}. Hence, trying to generalize the concept of
concavity, we included them in Definition~\ref{def}.
Moreover, since the concavity is antithetical to the
hyponormality (see \cite[p.\ 646]{C07}), it is reasonable to
expect that the third condition \eqref{wc}, resembling the
hyponormality, may yield some kind of concavity. Although at
first glance it is unclear why we use the term ``weakly
concave'', the following theorem sheds more light on this by
connecting, via Cauchy dual, the theory of weakly concave
operators with the theory of hyponormal operators.
   \begin{proposition} \label{bst-lemma}
Let $T \in \mathcal B(\hh)$ be a left-invertible operator.
Then the following conditions are equivalent{\em :}
   \begin{enumerate}
   \item[(i)] $T$ satisfies \eqref{wc},
   \item[(ii)] $T(T^*T)^{-1}T^* \Le T^*(T^*T)^{-1}T,$
   \item[(iii)] $P_{\ob{T}} \Le T^*(T^*T)^{-1}T,$
   \item[(iv)] the restriction of $T'$ to its range is hyponormal.
   \end{enumerate}
In particular, any concave operator is weakly concave.
   \end{proposition}
   \begin{proof}
That the conditions (i)-(iii) are equivalent is immediate
from the following identities (see Proposition~\ref{polur}):
   \begin{align} \label{tra-ta}
T'^*T' = (T^*T)^{-1} \quad \text{and} \quad
T^{\prime}T^*= P_{\ob{T}}.
   \end{align}
   To see that the conditions (i) and (iv) are
equivalent, note that $T'|_{\ob{T'}}$ is
hyponormal if and only if
   \begin{align*}
\|(T'|_{\ob{T'}})^*T'x\| \leqslant \|T'^2x\|,
\quad x \in \hh.
   \end{align*}
Using \eqref{tra-ta} and the equalities
$(T'|_{\ob{T'}})^* = P_{\ob{T'}}T'^*|_{\ob{T'}}$
and $\ob{T'}=\ob{T}$, we conclude that the
operator $T'|_{\ob{T'}}$ is hyponormal if and
only if
   \begin{align*}
\|T'T^*x\| \leqslant \|T'Tx\|, \quad x \in \hh.
   \end{align*}
The above inequality coincides with \eqref{wc}.
This shows that the conditions (i)-(iv) are
equivalent.

To show the remaining part, recall that any concave operator
satisfies \eqref{wc1} and \eqref{wc2}. Since the Cauchy dual
of a concave operator is hyponormal (see \cite[p.\
294]{Sh02}) and the restriction of a hyponormal operator to
its closed invariant subspace is hyponormal (see
\cite[Proposition~II.4.4]{Co1}), an application of the
implication (iv)$\Rightarrow$(i) completes the proof.
   \end{proof}
Our next goal is to show that the only
invertible weakly concave operators are unitary
operators (cf.\ Corollary~\ref{cly-uni}). In
fact, this is a consequence of a much more
general fact proved below.
   \begin{proposition} \label{cly-uni-prop}
Let $T \in \mathcal B(\hh)$ satisfy \eqref{wc1}
and let $T'|_{\hh_u'}$ be hyponormal, where
$\hh_u':=\bigcap_{n=0}^{\infty} \ob{T^{\prime
n}}$. Then the following statements are
valid{\em :}
   \begin{enumerate}
   \item[(i)] $T$ is unitary if  $T$ is invertible,
   \item[(ii)] $T$ is completely
non-normal if $T$ is analytic.
   \end{enumerate}
   \end{proposition}
   \begin{proof}
(i) Assume that $T$ is invertible. Since
$\partial \sigma(T) \subseteq \sigma_{ap}(T)$, we
infer from \eqref{wc1} that $\sigma(T) \subseteq
\partial \mathbb D$. Hence
   \begin{align} \label{fr-fy}
\sigma(T^*) = \sigma(T)^* \subseteq
\partial \mathbb D.
   \end{align}
Noting that the invertibility of $T$ implies the
invertibility of $T'$, we deduce that $\hh_u'=\hh$. From the
assumption that $T'|_{\hh_u'}$ is hyponormal, we see that
$T'$ is hyponormal. However, by invertibility of $T$,
$T'=T^{*-1},$ which by \cite[Proposition~II.4.4]{Co1} implies
that $T^*$ is hyponormal. Hence, by \eqref{fr-fy} and the
second corollary on \cite[Pg 473]{St-0}, $T$ is unitary.

(ii) Assume now that $T$ is analytic. Suppose to the contrary
that $T$ has a non-zero normal direct summand, say $N$. Then
by \eqref{wc1}, we have
   \begin{align*}
\sigma(N) = \sigma_{ap}(N) \subseteq \sigma_{ap}(T) \subseteq
\partial \mathbb D.
   \end{align*}
Hence, $N$ is unitary. This contradicts the fact that every
analytic operator is completely non-unitary.
   \end{proof}
The following corollary is immediate from
Propositions~\ref{bst-lemma} and
\ref{cly-uni-prop} and
\cite[Proposition~II.4.4]{Co1}.
   \begin{corollary} \label{cly-uni}
If $T \in \mathcal B(\hh)$ is weakly concave,
then the following statements are valid{\em :}
   \begin{enumerate}
   \item[(i)] $T$ is unitary if  $T$ is invertible,
   \item[(ii)] $T$ is completely
non-normal if $T$ is analytic.
   \end{enumerate}
   \end{corollary}
   \section{\label{Sec.3}Examples}
In this section, we show that there are weakly concave
operators which are not concave (see Example~\ref{przyk}
below). This is done by means of unilateral weighted shifts.
Recall that for a bounded sequence $\{w_n\}_{n = 0}^{\infty}$
of positive real numbers, there exists a unique operator
$W\in \mathcal B(\ell^2)$ (called a {\it unilateral weighted
shift with weight sequence $\{w_n\}_{n = 0}^{\infty}$}) such
that
   \begin{align*}
We_n \mathrel{\mathop:}= w_n e_{n + 1}, \quad n
\in \zbb_+,
   \end{align*}
where $\{e_n\}_{n = 0}^{\infty}$ stands for the standard
orthonormal basis of $\ell^2,$ the Hilbert space of square
summable complex sequences indexed by $\zbb_+.$ Note that $W$
is always analytic. Before providing the aforementioned
example, we need some auxiliary results. We begin by
characterizing those unilateral weighted shifts whose given
power is left-invertible and satisfies the
condition~\eqref{wc}.
   \begin{proposition} \label{W-w-c}
Let $W$ be a unilateral weighted shift with
weight sequence $\{w_n\}_{n=0}^{\infty}$ and let
$k$ be a positive integer. Then the following
statements are equivalent{\em :}
   \begin{enumerate}
   \item[(i)]
$W^k$ is left-invertible and satisfies
\eqref{wc},
   \item[(ii)]
$\displaystyle \inf_{_{n \in \zbb_+}}w_{n}
> 0$ and $\displaystyle \prod_{j=0}^{k-1}
\frac{w_{n+j}}{w_{k+n+j}} \geqslant 1$ for every
integer $n \geqslant k.$
   \end{enumerate}
   \end{proposition}
   \begin{proof}
First, observe that if $T\in \mathcal B(\hh),$ then for any
positive integer $m$, $T$ is left-invertible if and only if
$T^m$ is left-invertible. Hence, $W^k$ is left-invertible if
and only if $\inf_{n \in \zbb_+} w_n > 0$. This means that
there is no loss of generality in assuming that all powers of
$W$ are left-invertible. Noting that
   \allowdisplaybreaks
   \begin{align*}
W^k e_n & = \bigg(\prod_{j=0}^{k-1}
w_{n+j}\bigg) e_{k+n}, \quad n \in \zbb_+,
   \\
W^{*k} e_{k+n} & = \bigg(\prod_{j=0}^{k-1}
w_{n+j}\bigg) e_{n}, \quad n \in \zbb_+,
   \end{align*}
we deduce that
   \begin{align} \label{ws-as}
W^{*k}(W^{*k}W^k)^{-1}W^k e_n =
\prod_{j=0}^{k-1} \frac{w_{n+j}^2}{w_{k+n+j}^2}
e_{n}, \quad n \in \zbb_+.
   \end{align}
Clearly
   \begin{align} \label{ws-at}
P_{\ob{W^k}} e_n =
   \begin{cases}
0 & \text{if } n=0, \ldots k-1,
   \\
e_n & \text{if } n \Ge k.
   \end{cases}
   \end{align}
Combining \eqref{ws-as} and \eqref{ws-at} with
the equivalence (i)$\Leftrightarrow$(iii) of
Proposition~\ref{bst-lemma} completes the proof.
   \end{proof}
The next result we need is closely related to
\cite[Proposition 2.3]{AC}. It can be deduced
from \cite[Lemma~3.1]{ACJS-2} by applying the
fact that the (entry-wise) product of Hausdorff
moment sequences is a Hausdorff moment sequence
(see \cite[p.\ 179]{B-C-R}; see also
\cite[Subsection~6.2.4]{B-C-R} for the
definition of the Hausdorff moment sequence) and
by using the fundamental theorem of algebra (see
\cite[Theorem~V.3.19]{Hung}).
   \begin{proposition} \label{hau-zdrf}
Suppose that $p$ is a polynomial in one real
variable which has only negative real roots and
whose leading coefficient is positive. Then
$p(n)
> 0$ for all $n\in \zbb_+$ and $\{1/p(n)\}_{n=0}^{\infty}$
is a Hausdorff moment sequence.
   \end{proposition}
The following characterization of weakly concave
$m$-isometric unilateral weighted shifts is of
independent interest.
   \begin{proposition} \label{w-con-ch} Suppose that
$W$ is a unilateral weighted shift with weight
sequence $\{w_n\}_{n = 0}^{\infty}$ and $m \Ge
2$. Then the following statements are
equivalent{\em :}
   \begin{enumerate}
   \item[(i)] $W$ is a weakly concave
$m$-isometry,
   \item[(ii)] there exists a polynomial $p\in
\mathbb{R}[x]$ of degree at most $m-1$ such that
   \begin{gather} \label{star}
\text{$p(n) > 0$ and $w^2_n =
\frac{p(n+1)}{p(n)}$ for all $n \in \zbb_+,$}
   \\ \label{e-concave1}
\text{$p(n) \leqslant p(n+1)$ for all $n \in
\zbb_+,$}
   \\ \label{e-concave}
\text{$p(n)p(n+2) \leqslant p(n+1)^2$ for all $n
\in \zbb_+\backslash \{0\}.$}
   \end{gather}
   \end{enumerate}
Moreover, if $p$ is as in Proposition~{\em \ref{hau-zdrf}},
then \eqref{e-concave1} and \eqref{e-concave} hold.
   \end{proposition}
   \begin{proof}
By \cite[Theorem~2.1]{AL} (see also
\cite[Proposition~A.2]{JJS} and \cite[Theorem~3.4]{B-M-N}),
$W$ is an $m$-isometry if and only if there exists a
polynomial $p\in \mathbb{R}[x]$ of degree at most $m-1$ such
that \eqref{star} is valid. Therefore, there is no loss of
generality in assuming that $W$ is an $m$-isometry with $p\in
\mathbb{R}[x]$ of degree at most $m-1$ that satisfies
\eqref{star}. By \cite[Lemma~1.21]{Ag-St}, $W$ satisfies
\eqref{wc1}. Hence $0\notin \sigma_{ap}(W)$, or equivalently
$W$ is left-invertible. Moreover, $W$ is an expansion if and
only if $w_n \geqslant 1$ for every $n \in \zbb_+$. Hence,
applying \eqref{star} and Proposition~\ref{W-w-c} with $k=1$,
we see that the statements (i) and (ii) are equivalent.

Finally, if $p$ is as in Proposition~\ref{hau-zdrf}, then
$\{1/p(n)\}_{n=0}^{\infty}$ is a Hausdorff moment sequence
which as such is monotonically decreasing and logarithmically
convex (the former is equivalent to \eqref{e-concave1}, while
the latter implies \eqref{e-concave}).
   \end{proof}
We also need the following fact, which was proved for finite
systems of operators in \cite{CC}. For the reader's
convenience, we give its short proof below. This result is a
generalization of \cite[Proposition~3.6(ii)]{SA00}.
   \begin{proposition}[\mbox{\cite[Proposition 4.9]{CC}}] \label{gen-SA00}
Suppose that $T \in \mathcal B(\hh)$ is a concave
$m$-isometry for some $m\geqslant 2$. Then $T$ is a
$2$-isometry.
   \end{proposition}
   \begin{proof}
It follows from \cite[Lemma~1(b)]{R} that
   \begin{align} \label{pulim}
\|T^n h\|^2 \Le \|h\|^2 + n (\|Th\|^2 -
\|h\|^2), \quad h \in \hh, \, n\in \zbb_+.
   \end{align}
By \cite[p.\ 389]{Ag-St} (see also
\cite[Theorem~3.3]{JJS}), for every $h \in \hh$
the expression $\|T^n h\|^2$ is a polynomial
in\footnote{\label{fyto-1}i.e., there exists a
polynomial $p\in \mathbb{C}[z]$ such that
$p(n)=\|T^n h\|^2$ for all $n\in \zbb_+$.} $n$
of degree at most $m-1$. Combined with
\eqref{pulim}, this implies that $m \Le 2$,
which completes the proof.
   \end{proof}
We are now ready to construct a class of strict
$3$-isometric weakly concave operators. Note
that such operators cannot be concave which is
immediate from Proposition~\ref{gen-SA00}.
   \begin{example} \label{przyk}
Let $p\in \mathbb{R}[x]$ be the polynomial
defined by $p(x)=|x+z|^2$ for $x\in \mathbb{R}$,
where $z=u+iv \in \mathbb C$ with $u, v \in
\mathbb{R}$ such that $v \neq 0$. Since $p$ is
of degree $2$ and $p(n) > 0$ for all $n\in
\zbb_+$, the unilateral weighted shift $W$ with
weight sequence $\{w_n\}_{n = 0}^{\infty}$
defined by \eqref{star} is a $3$-isometry (see
\cite[Theorem~2.1]{AL} and
\cite[Proposition~A.2]{JJS}). Note that
   \begin{align} \label{expan-p}
\mbox{$p$ satisfies \eqref{e-concave1}
$\Longleftrightarrow$ $u \geqslant
-\frac{1}{2}.$}
   \end{align}
Straightforward computations show that
   \begin{align*}
\text{$p$ satisfies \eqref{e-concave} $\iff$
$\forall$ $n\Ge 1 \colon 2((n+u)^2-v^2) + 4(n+u)
+ 1 \Ge 0.$}
   \end{align*}
Combined with \eqref{expan-p} and
Proposition~\ref{w-con-ch}, this implies that
$W$ is weakly concave if and only if
   \begin{align} \label{gizo-nie}
\text{$u \geqslant -\frac{1}{2}$ and
$2((1+u)^2-v^2) + 4(1+u) + 1 \Ge 0.$}
   \end{align}
Since the second inequality in \eqref{gizo-nie}
is equivalent to
   \begin{align*}
2v^2 \Le 2(1+u)^2 + 4(1+u)+1,
   \end{align*}
we deduce that if $u \Ge 0$ and $v^2 \leqslant
7/2$, then $W$ is weakly concave.

Assume now that $u \geqslant -\frac{1}{2}.$
Although, in view of \eqref{gizo-nie}, $W$ need
not be weakly concave, quite surprisingly, $W^k$
turns out to be weakly concave for sufficiently
large $k$. To see this, we can argue as follows.
By \eqref{expan-p}, $W$ is expansive. Since $W$
is a 3-isometry, it follows from
\cite[Theorem~2.3]{J} that for every integer $k
\geqslant 1,$ $W^k$ is an expansive
$3$-isometry, and so by
\cite[Lemma~1.21]{Ag-St}, $T = W^k$ satisfies
\eqref{wc1} and \eqref{wc2}. Observe that the
second inequality in the condition (ii) of
Proposition~\ref{W-w-c} holds if and only if
   \begin{align} \label{moj-wnu}
p(n)p(n+2k) \Le p(n+k)^2, \quad n \Ge k.
   \end{align}
Straightforward computations show that
   \begin{align*}
\text{$p$ satisfies \eqref{moj-wnu} $\iff$
$\forall$ $n\Ge k \colon 2((n+u)^2-v^2) +
4(n+u)k + k^2 \Ge 0.$}
   \end{align*}
Now if $k$ is any positive integer such that
$k^2 \Ge 2v^2,$ then the above inequality holds,
and so $p$ satisfies \eqref{moj-wnu}.
Summarizing, if $u\Ge -\frac{1}{2}$ and $k \Ge
\sqrt{2}|v|$, then by Proposition~\ref{W-w-c},
$W^k$ is a weakly concave $3$-isometry. Note
that if $k$ is any positive integer, then $W^k$
is not a $2$-isometry. This follows from
\cite[Theorem~3.3]{JJS} by observing that the
identity
   \begin{align*}
\|(W^k)^n e_0\|^2 \overset{\eqref{star}} =
\frac{p(kn)}{p(0)} = \frac{(kn+u)^2 + v^2}{u^2 +
v^2}, \quad n\in \zbb_+,
   \end{align*}
implies that the expression $\|(W^k)^n e_0\|^2$
is a polynomial in $n$ of degree $2$ (see
footnote~\ref{fyto-1}). Hence, by
Proposition~\ref{gen-SA00}, $W^k$ is not concave
for any positive integer $k$.
   \hfill $\diamondsuit$
   \end{example}
   \section{\label{Sec.4}A spectral dichotomy}
Recall the well-known fact that the spectra of
an invertible operator $T \in \mathcal B(\hh)$
and its inverse $T^{-1} \in \mathcal B(\hh)$ are
related to each other and the correspondence is
given by
   \begin{align} \label{inv-fra}
\sigma(T^{-1})=\phi(\sigma(T)),
   \end{align}
where $\phi(\lambda)=1/\lambda$ is an automorphism of the
punctured plane $\mathbb C_{\bullet}:=\mathbb{C} \backslash
\{0\}$ (see \cite[Theorem~2.3.2]{Si} for the general version
of the spectral mapping theorem). A question worth
considering in the context of the present paper is to find
the counterpart of this simple fact for left-invertible
operators $T$. More specifically, if $ L$ is a left-inverse
of $ T$, then what is the relationship between the spectra of
$T$ and $L$. The purpose of this section is largely to answer
this question. It is worth noting that among the left
inverses of a left-invertible operator $T$ there exist such
as $T^{\prime *}$, where $T^{\prime}$ is the Cauchy dual of
$T$ (see \cite[Proposition~5.1]{Rich87} for concrete
examples).

Before stating the main result of this section,
observe that if $T\in \mathcal{B}(\mathcal{H})$
is left-invertible and $L\in
\mathcal{B}(\mathcal{H})$ is a left-inverse of
$T$, then
   \begin{align} \label{kukrtd}
\text{$0\in \sigma(T)$ if and only if $0\in
\sigma(L)$.}
   \end{align}
   \begin{theorem}[Spectral Dichotomy] \label{main}
Let $T \in \mathcal B(\hh)$ be a left-invertible
operator with a left-inverse $L \in \mathcal
B(\hh)$. Consider the automorphism $\phi$ of the
punctured plane $\mathbb{C}_{\bullet}$ given by
$\phi(z) = 1/{z}$ for $z \in \mathbb
C_{\bullet},$ and extend it to the Riemann
sphere $\mathbb C \cup \{\infty\}$ by setting
$\phi(0)=\infty$ and $\phi(\infty)=0.$ Then
   \begin{align} \label{tocobra}
\text{either $\mathbb C \setminus
\phi(\sigma_{l}(L))\subseteq
\sigma_{r}(T)\setminus \sigma_{l}(T)$ or
$\phi(\sigma(L)) =\sigma(T)$}
   \end{align}
   according as $0 \in \sigma(T)$ or $0 \notin
\sigma(T)$ $($see Fig.\ {\em \ref{fig6}}$)$. In
any case, $0 \notin \partial \sigma(L)$ and
   \begin{align} \label{bdd-inclusion}
\phi(\partial \sigma(L)) \subseteq \sigma_r(T).
   \end{align}
   \end{theorem}
   \begin{proof}
We begin by noting that
   \begin{align} \label{spec-id}
L-\lambda I = L (I-\lambda T) =
-\lambda L(T-\lambda^{-1}I), \quad \lambda \in
\mathbb C_{\bullet}.
   \end{align}
Second, there is no loss of generality in
assuming that $0\in \sigma(T)$ (because
otherwise $L=T^{-1}$, and so we can use
\eqref{inv-fra} to get $\phi(\sigma(L))
=\sigma(T)$). Thus, by \eqref{kukrtd}, $L$ is
not invertible. Since $T$ is left-invertible,
$T$ is not right-invertible. In turn, because
$L$ is right-invertible, $L$ is not
left-invertible. This yields
   \begin{align} \label{left-r}
0 \in \sigma_r(T) \setminus \sigma_l(T) \quad
\text{and} \quad 0 \in \sigma_l(L) \setminus
\sigma_r(L).
   \end{align}
We claim that
   \begin{align} \label{we-nudl}
\phi(\mathbb C_{\bullet} \setminus
\sigma_{l}(L)) \subseteq \sigma_{r}(T)\setminus
\sigma_{l}(T).
   \end{align}
For, assume that $\lambda \in \mathbb
C_{\bullet} \setminus \sigma_{l}(L)$. Then there
exists $L_1\in \mathcal B(\hh)$ such that
   \begin{align} \label{inv-ab}
L_1(L-\lambda I)=I.
   \end{align}
Hence, by \eqref{spec-id}, $T-\lambda^{-1}I$ is
left-invertible, meaning that $\lambda^{-1}\in
\mathbb C \setminus \sigma_{l}(T)$. Thus, to get
\eqref{we-nudl}, it suffices to show that
$\lambda^{-1} \in \sigma_{r}(T).$ Suppose, to
the contrary, that $\lambda^{-1} \in \mathbb C
\setminus \sigma_{r}(T)$. As a consequence,
$T-\lambda^{-1}I$ is invertible.
Right-multiplying \eqref{spec-id} by the inverse
of $T-\lambda^{-1}I$ yields
   \begin{align}  \label{sruf}
(L-\lambda I) (T-\lambda^{-1}I)^{-1} = - \lambda
L.
   \end{align}
Now left-multiplying \eqref{sruf} by $L_1$ leads
to
   \begin{align}  \label{srufi}
(T - \lambda^{-1} I)^{-1}
\overset{\eqref{inv-ab}} = L_1(L-\lambda I)(T -
\lambda^{-1}I)^{-1} = - \lambda L_1 L.
   \end{align}
Finally, left-multiplying \eqref{srufi} by $T -
\lambda^{-1} I$ gives
   \begin{align*}
I = \big(-\lambda (T - \lambda^{-1} I) L_1\big)
L.
   \end{align*}
This means that $L$ is left-invertible, showing
that $0 \in \mathbb{C} \backslash
\sigma_{l}(L)$, which contradicts
\eqref{left-r}. This justifies \eqref{we-nudl}.
Combined with \eqref{left-r}, this implies the
first inclusion in \eqref{tocobra}.

To prove \eqref{bdd-inclusion}, we first check
that $0 \notin \partial \sigma(L)$. Note that
   \begin{align} \label{bdd-inc}
0 \in \partial \sigma(L)
\Longleftrightarrow 0 \in \partial \sigma(L^*).
   \end{align}
Next, since $L$ is right-invertible and
   \begin{align*}
\sigma_{ap}(L^*)=\sigma_{l}(L^*)=\sigma_{r}(L)^{*},
   \end{align*}
we see that $0 \notin \sigma_{ap}(L^*)$.
Combined with \eqref{bdd-inc} and $\partial
\sigma(L^*) \subseteq \sigma_{ap}(L^*),$ this
implies that $0 \notin \partial \sigma(L)$.
Since, by \eqref{we-nudl}, $\phi(\mathbb
C_{\bullet} \setminus \sigma(L)) \subseteq
\sigma_{r}(T) \setminus \sigma_{l}(T),$ and the
right spectrum of any bounded linear operator is
closed, the inclusion $\phi(\partial \sigma(L))
\subseteq \sigma_r(T)$ now follows from the
continuity of $\phi.$ This completes the proof.
   \end{proof}
   \begin{figure}
\begin{tikzpicture}[scale=.8, transform shape]
\tikzset{vertex/.style =
{shape=circle,draw,minimum size=1em}}
\tikzset{edge/.style = {->,> = latex'}} \node[]
(z) at (-1,0) {}; \node[] (a) at (-2.7,-.4)
{$\sigma_{l}(L)$}; \node[] (ab) at (-.5,0.1) {$z
\stackrel{\phi}\longmapsto 1/z$}; \node[] (b) at
(3.2,0.7) {$\mathbb C \setminus
\phi(\sigma_{l}(L))$}; \node[] (ba) at (6.00,-1)
{$\sigma_r(T) \setminus \sigma_l(T)$}; \node[]
(bb) at (-3,.1) {$\centerdot$}; \node[] (d) at
(-2, 0) {}; \node[] (e) at (4.3, 0.4) {};
\node[] (f) at (5.2, 0.4) {}; \node[] (bbb) at
(3.2, .38) {$\centerdot$};
\def\Radius{1}
\draw[thick, dotted] (d) arc(0:360:\Radius) -- cycle;
\def\Radius{1.1}
\draw[thick, dotted] (e) arc(0:360:\Radius) -- cycle;
\def\Radius{2}
\draw[thick, dotted] (f) arc(0:360:\Radius) -- cycle;
   \end{tikzpicture}
   \caption{The relationship between the spectra
of a left-invertible operator $T$ and its
left-inverse $L,$ where $\phantom{}^\centerdot$
denotes the origin in the complex plane.}
\label{fig6}
   \end{figure}
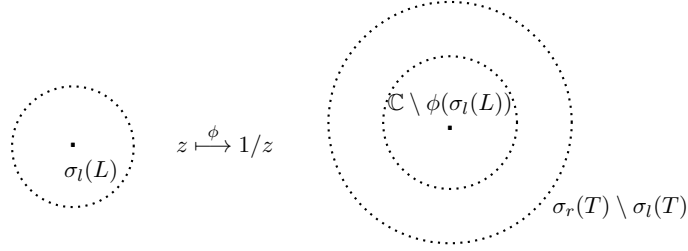
   \begin{corollary}
Let $T \in \mathcal B(\hh)$ be a left-invertible
operator with a left-inverse $L \in \mathcal
B(\hh)$. Then the spectral radius $r(L)$ of $L$
is positive and
   \begin{align} \label{furspit}
\text{either ${\mathbb D}_{1/r(L)} \subseteq
\sigma_r(T) \setminus \sigma_l(T)$ or $\sigma(T)
\subseteq \mathbb C \setminus {\mathbb
D}_{1/r(L)}$}
   \end{align}
according as $0 \in \sigma(T)$ or $0 \notin
\sigma(T)$. In particular, $r(T)r(L) \geqslant
1.$
   \end{corollary}
   \begin{proof}
By Theorem~\ref{main}, $0 \notin \partial
\sigma(L)$ and so $r(L) > 0$. According to the
definition of the spectral radius, $\mathbb
D_{1/r(L)} \cap \phi(\sigma(L)) = \emptyset$,
and thus the desired dichotomy is immediate from
Theorem \ref{main}. The inequality $r(T)r(L)
\geqslant 1$ follows from \eqref{furspit}.
   \end{proof}
   \begin{remark}
Since $L^nT^n=I$ holds for any positive integer
$n,$ one may derive $r(T)r(L) \geqslant 1$ also
from the Gelfand spectral radius formula.
   \hfill $\diamondsuit$
   \end{remark}

Note that if $T\in \mathcal{B}(\mathcal{H})$ is
a left-invertible operator, then by
Theorem~\ref{main} applied to the left-inverse
$L=T^{\prime *}$ of $T$, we get
   \begin{align} \label{dychhalf}
\text{$\mathbb C \setminus
\phi(\sigma_l(T^{\prime *})) \subseteq
\sigma_r(T) \setminus \sigma_l(T)$ provided $0
\in \sigma(T)$.}
   \end{align}
In general, the above inclusion is strict. We
discuss this in more detail in the context of
unilateral weighted shifts.
   \begin{example}  \label{egsimp}
Let $W\in \mathcal{B}(\ell^2)$ be a
left-invertible unilateral weighted shift with
positive weights $\{w_n\}_{n=0}^{\infty}$. Then
$W',$ the Cauchy dual of $W,$ is a
left-invertible unilateral weighted shift with
weights $\{w_n^{-1}\}_{n=0}^{\infty}$. As a
consequence of \cite[Theorem~6]{Rid70} and
\cite[Theorem~4]{Sh}, we get
   \begin{align} \label{kjura0}
\text{$\sigma_r(W)=\sigma(W) = \overline{\mathbb
D}_{r(W)}$ and
$\sigma_r(W^{\prime})=\sigma(W^{\prime}) =
\overline{\mathbb D}_{r(W^{\prime})}$.}
   \end{align}
In turn, \cite[Theorem~1]{Rid70} implies that
$\sigma_l(W)=\overline{\mathbb A(i(W), r(W))}$
(see \eqref{iT}). Since
   \begin{align*}
\sigma_l(W^{\prime *}) = \sigma_r(W^{\prime})^*
\overset{\eqref{kjura0}}= \overline{\mathbb
D}_{r(W^{\prime})},
   \end{align*}
and $0\in \sigma(W)$, we deduce from
\eqref{kjura0} that the condition
\eqref{dychhalf} takes the following form
   \begin{align*}
\mathbb D_{\frac{1}{r(W^{\prime})}} = \mathbb C
\setminus \phi\big(\overline{\mathbb
D}_{r(W^{\prime})}\big) = \mathbb C \setminus
\phi(\sigma_l(W^{\prime *})) \subseteq
\sigma_r(W) \setminus \sigma_l(W) = \mathbb
D_{i(W)}.
   \end{align*}
If we assume that equality holds in the above
inclusion, then $\frac{1}{r(W^{\prime})}=i(W),$
which is not true in general (see e.g.,
\cite[Example~2]{CT}, where $r(W^{\prime})= 2$
and $i(W) \Ge 2^{-3/4}$).
   \hfill $\diamondsuit$
   \end{example}
We conclude this section with some applications
of Theorem \ref{main} to the theory of Cauchy
dual operators.
   \begin{corollary} \label{last-p}
Let $T \in \mathcal B(\hh)$ be a left-invertible
operator such that $0 \in \sigma(T)$ and let
$T'$ be the Cauchy dual of $T.$ Then the
following assertions hold{\em :}
   \begin{enumerate}
   \item[(i)]
if $\sigma_r(T') \subseteq \overline{\mathbb
D},$ then ${\mathbb D} \subseteq \sigma_r(T)
\setminus \sigma_l(T),$
   \item[(ii)]
if $T$ is a contraction such that $\sigma_r(T')
\subseteq \overline{\mathbb D}$, then
$\sigma_r(T)=\sigma(T)=\overline{\mathbb D}$.
   \end{enumerate}
   \end{corollary}
   \begin{proof}
(i) Suppose that $\sigma_r(T') \subseteq \overline{\mathbb
D}$. Set $L=T^{\prime *}$. Then $L$ is a left-inverse of $T$
and
   \begin{align*}
\sigma_l(L) = \sigma_l(T^{\prime *}) =
\sigma_r(T^{\prime})^* \subseteq
\overline{\mathbb D}.
   \end{align*}
This combined with Theorem~\ref{main} yields
   \begin{align*}
\mathbb D=\mathbb C \setminus
\phi(\overline{\mathbb D}) \subseteq \mathbb C
\setminus \phi(\sigma_l(L)) \subseteq
\sigma_r(T) \setminus \sigma_l(T),
   \end{align*}
which justifies (i).

(ii) Assume now that $\|T\| \Le 1$ and
$\sigma_r(T') \subseteq \overline{\mathbb D}$.
Then using (i) and the well-known fact that the
spectral radius of a bounded linear operator is
at most its operator norm, we get
   \begin{align*}
{\mathbb D} \subseteq \sigma_r(T) \setminus
\sigma_l(T) \subseteq \sigma_r(T) \subseteq
\sigma(T) \subseteq \overline{{\mathbb D}}.
   \end{align*}
Since $\sigma_r(T)$ and $\sigma(T)$ are closed
subsets of $\mathbb{C}$, the proof is complete.
   \end{proof}
   \begin{remark} \label{dual-rmk}
It is well known and easy to prove that the Cauchy dual $T'$
of $T$ is left-invertible and $(T')'=T.$ This together with
\eqref{kukrtd} implies that $0 \in \sigma(T)$ if and only if
$0 \in \sigma(T').$ In particular, the roles of $T$ and $T'$
in Corollary~\ref{last-p} can be interchanged.
   \hfill $\diamondsuit$
   \end{remark}
 A solution of the wandering subspace problem for concave
operators given by Richter in \cite{R} can also be obtained
from the computations of the spectra of the Cauchy dual
operators (see \cite[Lemma~2.14]{C07}). In turn, these
computations can be derived from the Bunce result (see
\cite[Chapter II, $\S$12]{Co1}) on the existence of a rich
supply of non-zero $*$-homomorphisms of the $C^*$-algebra
generated by a hyponormal operator (see the proof of
\cite[Lemma~2.14]{C07}). Interestingly, the following fact,
which plays an important role in the proofs of
Theorems~\ref{wsp} and \ref{bst}, and which generalizes
\cite[Lemma 2.14(ii)]{C07}, can be deduced from the spectral
dichotomy.
   \begin{corollary} \label{sp-dual}
Let $T \in {\mathcal B}(\hh)$ be a
non-invertible expansion such that
$\sigma_{ap}(T) \subseteq \partial \mathbb D$
and let $T'$ be the Cauchy dual of $T.$ Then
   \begin{enumerate}
   \item[(i)] $\sigma(T) = \overline{\mathbb D} = \sigma(T')$,
   \item[(ii)] $\sigma_{ap}(T) = \partial {\mathbb D} =
\sigma_{ap}(T')$.
   \end{enumerate}
   \end{corollary}
   \begin{proof}
Recall the fact that for Hilbert space
operators, left spectrum and approximate point
spectrum coincide. Note that $T'$ is a
contraction because $T$ is an expansion.
Further, since $T$ is not invertible, $T'$ is
not invertible. Also, because the boundary of
the spectrum is contained in the approximate
point spectrum and $0 \in \sigma(T),$ we infer
from $\sigma_{ap}(T) \subseteq
\partial \mathbb D$ that $\sigma(T) =
\overline{\mathbb D}$ and
$\sigma_{ap}(T)=\partial \mathbb D.$ By applying
Corollary~\ref{last-p}(ii) to $T'$, we obtain
$\sigma(T')=\overline{\mathbb D}$ and $\partial
\mathbb D \subseteq \sigma_{ap}(T')$ (see Remark
\ref{dual-rmk}). An application of Corollary
\ref{last-p}(i) to $T'$ yields
   \begin{align*}
\mathbb D \subseteq \sigma(T') \setminus
\sigma_l(T') = \overline{\mathbb D} \setminus
\sigma_{ap}(T') \subseteq \mathbb D,
   \end{align*}
and hence $\sigma_{ap}(T')=\partial \mathbb D$.
This completes the proof.
   \end{proof}
   \section{\label{Sec.5}Proofs of Theorems~\ref{wsp} and \ref{bst}}
The proof of Theorem~\ref{wsp} is based on a
technique developed in \cite{S01}.
   \begin{proof}[Proof of Theorem~{\em \ref{wsp}}]
In view of Corollary~\ref{cly-uni}(i), we may
assume that $0 \in \sigma(T).$ Since the
subspace $\hh'_u$ (see \eqref{Hu}) is invariant
for $T'$, we can consider the operator
$R:=T'|_{\hh'_u} \in \mathcal B (\hh'_u).$ Note
that $T'$ is bounded from below and $R$ maps
$\hh'_u$ onto $\hh'_u$, hence $R$ is invertible.
By Corollary~\ref{sp-dual}(ii), we have
   \begin{align} \label{trzebu}
\sigma_{ap}(R) \subseteq \sigma_{ap}(T') =
\partial \mathbb D.
   \end{align}
Since $0 \notin \sigma(R)$ and $\partial
\sigma(R) \subseteq \sigma_{ap}(R)$, we deduce
from \eqref{trzebu} that $\sigma(R) \subseteq
\partial \mathbb D.$ Further, by Proposition~\ref{bst-lemma},
$R$ being a restriction of the hyponormal operator
$T'|_{\ob{T'}}$ to $\hh'_u$ is hyponormal. Hence, by
\cite[Second Corollary, p.\ 473]{St-0}, $R$ is unitary.
Recall that if $\mathcal M$ is a closed invariant subspace
for a contraction $A \in \mathcal B(\hh)$ such that
$A|_{\mathcal M}$ is unitary, then $\mathcal M$ reduces $A$
(see e.g., \cite[Lemma 2.5]{C1}). Applying this fact to
$A=T'$ ($T'$ is a contraction because $T$ is an expansion)
and $\mathcal M=\hh'_u,$ we deduce that $\hh'_u$ reduces
$T'$. By \cite[Proposition~2.2]{ACJS-2} applied to
$T^{\prime}$, $\hh'_u$ reduces $T$ and $T|_{\hh'_u} =
R^{\prime}$. Consequently, $T|_{\hh'_u}$ is unitary and so
$T(\hh'_u)=\hh'_u.$ This in turn implies that $\hh'_u
\subseteq \hh_u$ (see \eqref{Hu}). On the other hand, by
\cite[Proposition~2.7(i)]{S01} we have
   \begin{align*}
(\hh'_u)^{\perp}=\bigvee_{n \in \zbb_+} T^n(\ker
T^*).
   \end{align*}
Since $T|_{\hh'_u}$ is unitary, we conclude that
$T|_{\hh \ominus \hh'_u}$ has the wandering
subspace property. This completes the proof.
   \end{proof}
   \begin{corollary}[{\cite[Theorem~3.6(A)]{S01}}] \label{thm-gen}
If $T \in \mathcal B(\hh)$ is a concave
operator, then $T$ has the Wold-type
decomposition, that is, $\mathcal{H}_{u}$
reduces $T$, $T|_{\mathcal{H}_{u}}$ is unitary
and
   \begin{align*}
\mathcal{H} = \mathcal{H}_{u} \oplus \bigvee_{n
\in \zbb_+} T^n(\ker T^*),
   \end{align*}
where $\hh_u$ is as in {\em \eqref{Hu}}.
   \end{corollary}
   \begin{proof}
It follows from \cite[Proposition 3.4]{S01} that
$\hh_u$ reduces $T$ to a unitary operator, so by
\eqref{wsp-cnu} and Theorem \ref{wsp} we
conclude that $\hh_u = \hh'_u$, which completes
the proof.
   \end{proof}
   \begin{corollary}
Any weakly concave operator $T \in \mathcal B(\hh)$ which is
analytic has the Wold-type decomposition.
   \end{corollary}
The following lemma is needed in the proof of
Theorem~\ref{bst}.
   \begin{lemma} \label{Jan}
Suppose that $T \in \mathcal B(\hh)$ is a
finitely cyclic operator. Then
   \begin{align} \label{rare}
\sigma_{e}(T) & \subseteq \sigma_{ap}(T),
   \\ \label{rara}
\mathrm{ind}_{T}(\lambda) & = - \dim \ker
(T^*-\bar \lambda I), \quad \lambda\in \mathbb C
\setminus \sigma_{ap}(T).
   \end{align}
   \end{lemma}
   \begin{proof} Suppose that
$\lambda\in \mathbb C \setminus \sigma_{ap}(T)$.
Then $T-\lambda I$ is left-invertible, so the
space $\ob{T-\lambda I}$ is closed and $\dim
\ker (T-\lambda I)=0$. According to
\eqref{fundim}, $\dim \ker (T^*-\bar \lambda I)<
\infty$. This means that the operator $T-\lambda
I$ is Fredholm, so $\lambda \in \mathbb C
\setminus \sigma_{e}(T).$ This proves
\eqref{rare}. Moreover, we have
   \begin{align*}  \notag
\mathrm{ind}_{T}(\lambda)& = \dim \ker
(T-\lambda I) - \dim \ker (T^*-\bar \lambda I)
   \\
&= - \dim \ker (T^*-\bar \lambda I).
   \end{align*}
   This completes the proof.
   \end{proof}
   \begin{proof}[Proof of Theorem~{\em \ref{bst}}]
By \eqref{fundim}, $\ker T^*$ is finite
dimensional. Combined with
\cite[Proposition~2.7(ii)]{S01} and the
hypothesis that $T$ is analytic, this implies
that $T'$ is finitely cyclic with the cyclic
space $\ker T^*$. Moreover, since the range of
$T'$ is closed, we have
   \begin{align} \label{dlapor}
\bigvee_{n \in \zbb_+}
(T'|_{\ob{T'}})^n\big(T'(\ker T^*)\big) = T'
\Big(\bigvee_{n \in \zbb_+} T'^n(\ker T^*)\Big)
=\ob{T'}.
   \end{align}
It follows that $T'|_{\ob{T'}}$ is finitely
cyclic with the cyclic space $T'(\ker T^*)$. By
Proposition~\ref{bst-lemma}, $T'|_{\ob{T'}}$ is
hyponormal. Hence, by the Berger-Shaw Theorem
(see \cite[Theorem~IV.2.1]{Co1}),
$T'|_{\ob{T'}}$ has a trace class
self-commutator. Since $T'$ is a finite rank
perturbation of $T'|_{\ob{T'}} \oplus 0$, where
$0$ is the zero operator on $\ker{T^{\prime *}}
= \ker{T^*}$, and the set of finite rank
operators is a $*$-ideal, one can see that $T'$
has a trace class self-commutator. Observing
that for any left-invertible operator $A \in
\mathcal B(\hh)$ the following operator
identities hold (cf.\ the proof of
\cite[Proposition~2.21]{C07}):
   \begin{align} \label{dwuwz}
   \begin{gathered}
[A^*,A]A \overset{\eqref{tra-ta}}= -A^*A[A'^*,A']AA^*A,
   \\
[A^*, A] \overset{\eqref{tra-ta}} = ([A^*,A]A)A'^* +
[A^*,A]P,
   \end{gathered}
   \end{align}
where $P$ denotes the orthogonal projection of
$\hh$ onto $\ker A^*,$ and using the fact that
the set of trace class operators is an ideal
(see \cite[Theorem~3.6.6]{Si}), we deduce that
$T$ has a trace class self-commutator.

To prove the ``moreover'' part we proceed as
follows. First, we show that
   \begin{align} \label{cucu}
\text{$\sigma(S) = \bar{\mathbb D}$ \; and \;
$\sigma_{e}(S) = \sigma_{ap}(S) =\partial
\mathbb D$ \; for \; $S \in \{T,T'\}.$}
   \end{align}
Since $T$ is analytic, we see that $0\in
\sigma(T)$. In view of the above discussion, $T$
and $T'$ are finitely cyclic. Let $S$ be any of
the operators $T$ or $T'.$ It follows from
Corollary~\ref{sp-dual} and Lemma~\ref{Jan} that
$\sigma(S)=\overline {\mathbb D}$ and
$\sigma_{e}(S) \subseteq \sigma_{ap}(S)=\partial
{\mathbb D}$. However, if $\lambda \in \partial
{\mathbb D}=\partial \sigma(S)$, then by
\cite[Theorem~XI.6.8]{Co} either $\lambda$ is an
isolated point of $\sigma(S)$ (which is not the
case) or $\lambda \in \sigma_{e}(S)$, meaning
that $\partial {\mathbb D} \subseteq
\sigma_{e}(S)$. Putting all this together, we
conclude that $\sigma_{e}(S) = \sigma_{ap}(S)
=\partial \mathbb D$, which completes the proof
of \eqref{cucu}.

Next, applying \eqref{cucu} and the fact that
$\mathrm{ind}_S(\cdot)$ is constant on connected
components of $\mathbb C \backslash
\sigma_{e}(S)$ (use
\cite[Corollary~XI.3.13]{Co}), we deduce that
   \begin{align} \notag
\mathrm{ind}_{T}(\lambda) & =
\mathrm{ind}_{T}(0)
   \\ \notag
& \hspace{-1ex} \overset{\eqref{rara}} = - \dim
\ker T^*
   \\ \notag
& = - \dim \ker T^{\prime *}
   \\ \label{fafa}
& \hspace{-1ex} \overset{\eqref{rara}} =
\mathrm{ind}_{T'}(0) =
\mathrm{ind}_{T'}(\lambda), \quad \lambda\in
\mathbb D.
   \end{align}
It follows from \cite[Theorem~8.1]{C-P77} that
for any $\lambda \in \mathbb C \backslash
\sigma_{e}(S)$, the Pincus principal function
$\mathcal P_S(\cdot)$ of a completely non-normal
operator $S$ with a trace class self-commutator
is equal to $\mathrm{ind}_{S}(\lambda)$ a.e.\
with respect to the planar Lebesgue measure in a
neighbourhood of $\lambda$. By
Corollary~\ref{cly-uni}(ii) and
\cite[Proposition~2.2]{ACJS-2}, $T$ and $T'$ are
completely non-normal, and thus by \eqref{cucu}
and \eqref{fafa}, we deduce that
   \begin{align} \label{leba}
\mathcal P_T(\lambda)= \mathcal P_{T'}(\lambda)=
- \dim \ker T^* \cdot \chi_{\mathbb D}(\lambda)
\quad \text{for a.e.\ $\lambda\in \mathbb C,$}
   \end{align}
where $\chi_{\mathbb D}$ stands for the
characteristic function of $\mathbb D.$ Since
the Carey-Pincus trace formula for polynomials
in $S$ and $S^*$ is valid for completely
non-normal operators $S$ with a trace class
self-commutator (see \cite[Theorem~5.1]{C-P77};
see also \cite[Theorem, p.~148]{H-H73}), the
proof is complete.
   \end{proof}
   \begin{remark}
Let us mention that the identity $\sigma_{e}(S)
=\partial \mathbb D$ which appears in
\eqref{cucu} can be proved without using
\cite[Theorem~XI.6.8]{Co}. Indeed, as in the
proof of \eqref{cucu}, we see that
$\sigma_{e}(S) \subseteq \sigma_{ap}(S) =
\partial {\mathbb D}$, where $S \in \{T,T'\}$.
Suppose, to the contrary, that $\sigma_{e}(S)$
is a proper subset of $\partial {\mathbb D}$.
Then
   \begin{align*}
\mathbb D\cup (\mathbb{C} \setminus
\overline{\mathbb{D}}) \subseteq \varOmega,
   \end{align*}
where $\varOmega$ is the connected component of
$\mathbb C \backslash \sigma_{e}(S)$. Since
$\mathrm{ind}_S(\cdot)$ is constant on connected
components of $\mathbb C \setminus
\sigma_{e}(S)$ and, by
Corollary~\ref{sp-dual}(i),
$\mathrm{ind}_{S}(\lambda)=0$ whenever $\lambda
\in \mathbb{C} \setminus \overline{\mathbb{D}}$,
we conclude that (recall that $\dim \ker T^* =
\dim \ker T^{\prime *}$)
   \begin{align*}
- \dim \ker T^* \overset{\eqref{rara}}=
\mathrm{ind}_{S}(0) = \mathrm{ind}_{S}(\lambda)
= 0, \quad \lambda \in \mathbb{C} \setminus
\overline{\mathbb{D}}.
   \end{align*}
This implies that $T$ is invertible, which
contradicts $0\in \sigma(T)$.
   \hfill $\diamondsuit$
   \end{remark}
   \section{\label{Sec.6}Generalizations to the classes $\mathcal A_{k}$}
   Motivated by the study of weakly concave
operators, we introduce and discuss the
following related classes of operators.
   \begin{definition} \label{duffa}
Let $T\in \mathcal B(\hh)$. We say that $T$ is
{\em of class} $\mathcal A_{\infty}$ if
   \begin{enumerate}
   \item[(i)] $\sigma_{ap}(T) \subseteq \partial \mathbb D$,
   \item[(ii)] $T$ is an expansion,
   \item[(iii$_{\infty}$)] $T'|_{\hh_u'}$ is hyponormal
with $\hh_u'=\bigcap_{n=0}^{\infty}
\ob{T^{\prime n}}$.
   \end{enumerate}
Given $k \in \zbb_+,$ we say that $T$ is {\em of
class} $\mathcal A_{k}$ if (i), (ii) and
(iii$_{k}$) hold, where
 \begin{enumerate}
   \item[(iii$_{k}$)] $T'|_{\ob{T'^k}}$ is hyponormal.
   \end{enumerate}
   \end{definition}
In view of the second paragraph of the proof of
Proposition~\ref{bst-lemma}, the class $\mathcal A_0$
contains all concave operators. It follows from
Proposition~\ref{bst-lemma} that the class $\mathcal A_1$
consists exactly of weakly concave operators. Since the
restriction of a hyponormal operator to its closed invariant
subspace is again hyponormal, we conclude~that
   \begin{align} \label{stukk}
\{\text{\em concave operators}\} \subseteq
\mathcal A_{k} \subseteq \mathcal A_{k+1}
\subseteq \mathcal A_{\infty}, \quad k\in
\zbb_+.
   \end{align}
It turns out that all inclusions in
\eqref{stukk} are proper if $\hh$ is infinite
dimensional (see Example~\ref{separ-1}). If
$\hh$ is finite dimensional, then for every
$k\in \zbb_+ \cup \{\infty\}$, $\mathcal A_k$
consists exactly of unitary operators (see
Proposition~\ref{cly-uni-prop}(i)).

An examination of the proof of Theorem~\ref{wsp}
(using Proposition~\ref{cly-uni-prop}(i) in
place of Corollary~\ref{cly-uni}(i)) shows that
it is true for operators of class $\mathcal
A_{k}$, where $k\in \zbb_+ \cup \{\infty\}$.
   \begin{theorem} \label{Theorem2.7*}
Suppose that $T \in \mathcal B(\hh)$ is of class
$\mathcal A_{\infty}$. Let $\hh_u$ and $\hh'_u$
denote the closed vector subspaces of $\hh$
given by \eqref{Hu}. Then $\hh'_u \subseteq
\hh_u$ and $\hh'_u$ reduces $T.$ Moreover,
   \begin{align*}
T = U \oplus S \text{ with respect to the
decomposition } \hh = \hh'_u \oplus
(\hh'_u)^{\perp},
   \end{align*}
where $U \in \mathcal B(\hh'_u)$ is unitary and
$S \in \mathcal B((\hh'_u)^{\perp})$ has the
wandering subspace property.
   \end{theorem}
   \begin{remark}
Regarding Theorems~\ref{wsp} and
\ref{Theorem2.7*}, we make the following
observation. Let $T \in {\mathcal B}(\hh)$ be an
analytic expansion with $\sigma_{ap}(T)
\subseteq \partial \mathbb D$ and let $T'$ be
the Cauchy dual of $T.$ Set $R:=T'|_{\hh'_u}$.
Since the analyticity of $T$ implies that $0\in
\sigma(T)$, it can be concluded from the proof
of Theorem~\ref{wsp} that $R$ is an invertible
contraction such that $\sigma_{ap}(R) \subseteq
\partial \mathbb D$ (see \eqref{trzebu}),
which together with the inclusion $\partial
\sigma(R) \subseteq \sigma_{ap}(R)$ implies that
$\sigma(R) \subseteq \partial \mathbb D$ (the
reader is referred to \cite[Theorem 7]{St},
where a non-unitary contraction with spectrum
being a nowhere dense perfect set of measure
zero on the unit circle has been constructed).
Thus $R$ is unitary if and only if $R$ is
hyponormal. In particular, this is the case if
the restriction of $T'$ to the range of some
power of it is hyponormal (see \eqref{stukk} and
Theorem~\ref{Theorem2.7*}).
   \hfill $\diamondsuit$
   \end{remark}
In turn, Theorem~\ref{bst} remains valid for
operators of class $\mathcal A_{k}$, where $k\in
\zbb_+$.
   \begin{theorem} \label{bst-Jan}
Let $T \in \mathcal B(\hh)$ be an analytic
finitely cyclic operator of class $\mathcal
A_{k}$ for some $k\in\zbb_+.$ Then $T$ and $T'$
have trace class self-commutators. Moreover, for
any complex polynomials $p, q$ in two complex
variables $z$ and $\bar z$,
   \begin{align*}
\mathrm{tr}\,[p(T', T'^*), q(T', T'^*)] &=
\mathrm{tr}\,[p(T, T^*), q(T, T^*)]
   \\
&= \frac{1}{\pi}\dim \ker T^* \int_{\D}
\Big(\frac{\partial p}{\partial \bar
z}\frac{\partial q}{\partial z} - \frac{\partial
p}{\partial z}\frac{\partial q}{\partial \bar
z}\Big) (z, \bar z) dA(z),
   \end{align*}
where $dA$ denotes the Lebesgue planar measure.
   \end{theorem}
   \begin{proof}
According to Proposition~\ref{bst-lemma}, the
case $k=1$ is precisely Theorem~\ref{bst} (by
\eqref{stukk}, this implies the case $k=0$).
Assume now that $k \Ge 2.$ Since $T$ is finitely
cyclic, we infer from \eqref{fundim} that $\dim
\ker T^* < \infty$. By the analyticity of $T$,
$\hh_u=\{0\}$, so by \cite[Proposition
2.7(ii)]{S01} $T'$ is finitely cyclic with the
cyclic space $\ker T^*$. Since the range of
$T^{\prime k}$ is closed, we deduce that
$R_k:=T'|_{\ob{T^{\prime k}}}$ is finitely
cyclic with the cyclic space $T^{\prime k}(\ker
T^*)$ (cf.\ \eqref{dlapor}). By (iii$_{k}$),
$R_k$ is hyponormal, so by
\cite[Theorem~IV.2.1]{Co1}, $[R_k^*,R_k]$ is of
trace class. It follows from \cite[Lemma
2.1]{S01} that
   \begin{align*}
\ob{T^{\prime k}}^{\perp}=\ker{T^{\prime *
k}}=\bigvee_{j=0}^{k-1} T^j(\ker T^*),
   \end{align*}
which yields $\dim \ob{T^{\prime k}}^{\perp} <
\infty$. This implies that $T'$ is a finite rank
perturbation of the operator $R_k \oplus 0$
which has a trace class self-commutator. Thus,
the self-commutator of $T'$ and consequently
that of $T$ are of trace class
(see~\eqref{dwuwz}). It is also easy to see that
\eqref{cucu} holds (with the same proof). Next,
arguing as in last paragraph of the proof of
Theorem \ref{bst} (with
Proposition~\ref{cly-uni-prop}(ii) in place of
Corollary~\ref{cly-uni}(ii)), we conclude that
\eqref{leba} holds. Finally, applying the
Carey-Pincus trace formula completes the proof.
   \end{proof}
   \begin{corollary} \label{bdfir}
Suppose that $\hh\neq \{0\}$ and $T \in \mathcal B(\hh)$ is
an analytic finitely cyclic operator satisfying the
conditions {\em (i)} and {\em (ii)} of Definition~{\em
\ref{duffa}}. Then
   \begin{align} \label{toraz}
   \begin{gathered}
\sigma_{le}(T) = \sigma_{re}(T)=\sigma_{e}(T)= \sigma_{ap}(T)
= \partial \mathbb D,
   \\
\sigma_{le}(T')=\sigma_{re}(T')=\sigma_{e}(T')=
\sigma_{ap}(T') = \partial \mathbb D,
   \end{gathered}
   \end{align}
where $\sigma_{le}(\cdot)$ and $\sigma_{re}(\cdot)$ denote
the left essential and the right essential spectrum,
respectively. Moreover, $\dim \ker(T^*)$ is a positive
integer and
   \begin{align} \label{toroz}
\mathrm{ind}_{T}(\lambda) = \mathrm{ind}_{T'}(\lambda) =
   \begin{cases}
- \dim \ker T^* & \text{if } \lambda \in \mathbb
D,
\\[1ex]
0 & \text{if } \lambda \in \mathbb C \setminus
\overline{\mathbb D}.
   \end{cases}
   \end{align}
   \end{corollary}
   \begin{proof}
Look carefully at \cite[Theorem~XI.6.8]{Co} and the proofs of
\eqref{cucu} and \eqref{fafa}. That $\dim \ker(T^*) \Ge 1$
can be inferred from \eqref{toraz}, \eqref{toroz} and the
fact that the Fredholm index $\mathrm{ind}_S(\cdot)$ of any
$S\in \mathcal B(\hh)$ is constant on connected components of
$\mathbb C \backslash \sigma_{e}(S).$
  \end{proof}
The next result generalizes \cite[Corollary~2.29]{C07}, which
was stated for concave operators, to operators of class
$\mathcal A_k$ with $k\in \zbb_+$ (see \eqref{stukk}). This
is also related to \cite{Co67}.
   \begin{corollary} \label{bst-chrup}
Suppose that $\hh\neq \{0\}$ and $T \in \mathcal B(\hh)$ is
an analytic finitely cyclic operator of class $\mathcal
A_{k},$ where $k\in\zbb_+.$ Then $n: =\dim \ker T^*$ is a
positive integer and $T$ is unitarily equivalent to a compact
perturbation of $U^n_+$, where $U_+$ is the unilateral shift
of multiplicity $1.$
   \end{corollary}
   \begin{proof}
By Corollary~\ref{bdfir}, $n$ is a positive integer. Next,
note that $U_+^n$ is a concave operator which is analytic and
finitely cyclic with $\dim\ker(U_+^{n})^{*}=n$. Applying
Corollary~\ref{bdfir} to the operators $T$ and $U_+^n$, we
conclude that
   \begin{gather} \label{BDF-equ}
\sigma_{e}(T)=\sigma_{e}(U_+^n)=\partial \mathbb
D,
   \\ \label{BDF-eqv}
\mathrm{ind}_{T}(\lambda) =
\mathrm{ind}_{U_+^n}(\lambda), \quad \lambda\in
\mathbb C \setminus
\partial \mathbb D.
   \end{gather}
By Theorem~\ref{bst-Jan}, the operators $T$ and
$U_+^n$ have compact self-commutators. This
together with \eqref{BDF-equ}, \eqref{BDF-eqv}
and \cite[Theorem~2]{Doug80} completes the
proof.
   \end{proof}
The example below shows that if $\hh$ is
infinite dimensional, then all the inclusions in
\eqref{stukk} are proper, i.e.,
   \begin{align} \label{uno-4}
\{\text{\em concave operators}\} \varsubsetneq
\mathcal A_{k} \varsubsetneq \mathcal A_{k+1}
\varsubsetneq \mathcal A_{\infty}, \quad k\in
\zbb_+,
   \end{align}
and Theorem~\ref{bst-Jan} is not true for
operators of class $\mathcal A_{\infty}$.
   \begin{example}[Example~\ref{egsimp} continued] \label{separ-1}
As in Example~\ref{egsimp}, we assume that the
unilateral weighted shift $W$ is
left-invertible. Clearly, $W$ is cyclic with the
cyclic vector $e_0$ and is expansive if and only
if
   \begin{align} \label{uno-1}
\text{$w_n \Ge 1$ for all $n\in \zbb_+.$}
   \end{align}
Fix $k\in \zbb_+$. Since $W$ and $W'$ are
unilateral weighted shifts, we have
   \begin{align} \label{pppp}
\ob{W^k} = \ob{W^{\prime k}} =
\bigvee_{n=k}^{\infty} \mathbb C e_n.
   \end{align}
This implies that $W$ and $W'$ are analytic. By
\eqref{pppp}, the operator $W'|_{\ob{W^{\prime
k}}}$ is unitarily equivalent to the unilateral
weighted shift with weights
$\{w_{k+n}^{-1}\}_{n=0}^{\infty}$. Hence, by
\cite[Proposition~II.6.6]{Co1},
$W'|_{\ob{W^{\prime k}}}$ is hyponormal if and
only if
   \begin{align} \label{uno-2}
\text{the sequence $\{w_{k+n}\}_{n=0}^{\infty}$
is monotonically decreasing.}
   \end{align}
Consider now the inequality
   \begin{align} \label{uno-3}
\limsup_{n\to \infty} w_n\Le 1,
   \end{align}
and note that the following implication is
valid.
   \begin{align} \label{lisisr}
   \begin{minipage}{65ex}
{\em If \eqref{uno-3} holds, then $r(W)\Le 1$.}
   \end{minipage}
   \end{align}
Indeed, if $\varepsilon > 0$, then $\sup_{n\Ge
l} w_n \Le 1+\varepsilon$ for some integer $l\Ge
1$, so
   \begin{align*}
\sup_{n\Ge 0} w_n \cdots w_{n+m-1} \Le M
(1+\varepsilon)^m, \quad m \Ge l+1,
   \end{align*}
with $M:= \max\big\{1, w_{l-1}, w_{l-1} \cdot
w_{l-2}, \ldots, w_{l-1} \cdots w_{0}\big\}.$
Together with Gelfand's spectral radius formula,
this implies that
   \begin{align*}
r(W) = \lim_{m \to \infty} \|W^m\|^{1/m}
=\lim_{m\to \infty} (\sup_{n\Ge 0} w_n \cdots
w_{n+m-1})^{1/m} \Le 1+ \varepsilon.
   \end{align*}
which yields $r(W)\Le 1$.

Combining \eqref{lisisr} and
Remark~\ref{rmk-Jan} with the analyticity of
$W'$, we obtain the following sufficient
condition for membership in the
class\footnote{Since \eqref{uno-1} and
\eqref{uno-3} imply that $\lim_{n\to \infty} w_n
= 1$, the implication \eqref{uno-6} also follows
from \cite[Proposition~15]{Sh} and
Remark~\ref{rmk-Jan}.} $\mathcal A_{\infty}$.
   \begin{align} \label{uno-6}
   \begin{minipage}{65ex}
{\em If {\em \eqref{uno-1}} and {\em
\eqref{uno-3}} hold, then $W$ is of class
$\mathcal A_{\infty}$.}
   \end{minipage}
   \end{align}
Moreover, the following necessary and sufficient
condition for membership in the class $\mathcal
A_k$ with $k \in \zbb_+$ can be deduced from
\cite[Proposition~15]{Sh} and
Remark~\ref{rmk-Jan}.
   \begin{align} \label{uno-5}
   \begin{minipage}{65ex}
\text{\em $W$ is of class $\mathcal A_k$ if and
only if {\em \eqref{uno-1}, \eqref{uno-2}} and
{\em \eqref{uno-3}} hold.}
   \end{minipage}
   \end{align}
It is now easy to deduce from \eqref{uno-6} and
\eqref{uno-5} that
   \begin{align*}
\mathcal A_{k} \varsubsetneq \mathcal A_{k+1}
\varsubsetneq \mathcal A_{\infty}, \quad k\in
\zbb_+.
   \end{align*}
To obtain the operator of class $\mathcal A_0$
that is not concave, we proceed as follows. By
\cite[Proposition~A.2]{JJS}, the unilateral
weighted shift $W$ with weight sequence
$\big\{\frac{n+2}{n+1}\big\}_{n=0}^{\infty}$ is
a strict $3$-isometric expansion. Since $W'$ is
the unilateral weighted shift with weight
sequence
$\big\{\frac{n+1}{n+2}\big\}_{n=0}^{\infty}$
which is increasing, $W'$ is hyponormal, meaning
that $W$ is of class $\mathcal A_0$. However, by
Proposition~\ref{gen-SA00}, $W$ is not concave.
This justifies all the strict inclusions
in~\eqref{uno-4}.

The modulus of the self-commutator $[W^*,W]$ of $W$ is the
diagonal operator (relative to $\{e_n\}_{n=0}^{\infty}$) with
diagonal elements $\{|w_n^2 - w_{n-1}^2|\}_{n=0}^{\infty}$,
where $w_{-1}:=0$. Suppose that $k\in \zbb_+$ and $W$ is of
class $\mathcal A_k$. Then, by \eqref{uno-5}, we~have
   \allowdisplaybreaks
   \begin{align*}
\sum_{n=k+1}^{\infty} |w_n^2 - w_{n-1}^2| &
\overset{\eqref{uno-2}}= \sum_{n=k+1}^{\infty}
(w_{n-1}^2 - w_n^2)
   \\
& \hspace{1ex}=\lim_{m\to \infty}
\sum_{n=k+1}^{m} (w_{n-1}^2 - w_n^2)
   \\
& \hspace{1ex} = \lim_{m\to \infty} (w_{k}^2 -
w_m^2)
   \\
& \hspace{-2.5ex}\overset{\eqref{uno-1} \&
\eqref{uno-3}}= w_{k}^2 -1 < \infty,
   \end{align*}
which means that $[W^*,W]$ is of trace class.
The above reasoning can be used to compute the
trace $\mathrm{tr}\, [W^*,W]$ of the
self-commutator $[W^*,W]$:
   \begin{align} \label{ticefurm}
\mathrm{tr}\, [W^*,W] = \sum_{n=0}^{\infty}
(w_n^2 - w_{n-1}^2) = \lim_{m\to \infty} w_m^2
\overset{\eqref{uno-1} \& \eqref{uno-3}}= 1.
   \end{align}
Certainly, the fact that $W$ has a trace class
self-commutator as well as the formula
\eqref{ticefurm} are particular cases of what is
in Theorem~\ref{bst-Jan} with the polynomials
$p(z,\bar z) = \bar z$ and $q(z,\bar z) = z$.

Observe that if $W$ satisfies \eqref{uno-1} and
\eqref{uno-3}, then $W$ is of class $\mathcal A_{\infty}$ and
the self-commutator $[W^*,W]$ of $W$ is compact. That $W$ is
of class $\mathcal A_{\infty}$ follows from \eqref{uno-6}. In
turn, by our assumptions $\lim_{n\to \infty} w_n=1$, which
implies that $[W^*,W]$, being the diagonal operator with
diagonal elements $\{w_n^2 - w_{n-1}^2\}_{n=0}^{\infty}$
converging to $0$, is compact.

Now we construct a unilateral weighted shift $W$
of class $\mathcal A_{\infty}$ such that the
self-commutator $[W^*,W]$ of $W$ is not of trace
class (recall that $W$ is analytic and cyclic).
Let $\{w_n\}_{n=0}^{\infty}$ be defined by
   \allowdisplaybreaks
   \begin{align*}
\underset{5 \text{ times}}{\underbrace{\alpha_1,
\alpha_2, \alpha_1, \alpha_2, \alpha_1}},
\ldots, \underset{2^{l+1} + 1\text{
times}}{\underbrace{\alpha_l, \alpha_{l+1},
\ldots, \alpha_l, \alpha_{l+1}, \alpha_l}},
\ldots,
   \end{align*}
where $\alpha_l = \sqrt{1+\frac 1{2^l}}$ for
$l=1,2,\ldots$. Straightforward computations
show that
   \begin{align*}
\sum_{n=0}^{\infty} |w_n^2 - w_{n-1}^2| = \infty,
   \end{align*}
which means that $[W^*,W]$ is not of trace
class. Since $W$ satisfies \eqref{uno-1} and
\eqref{uno-3}, we infer from \eqref{uno-6} that
$W$ is of class $\mathcal A_{\infty}$.
   \hfill $\diamondsuit$
   \end{example}
As shown in Example~\ref{separ-1} (see its last
paragraph), Theorem~\ref{bst-Jan} ceases to be
true for operators of class $\mathcal
A_{\infty}$. However, in this example $\mathcal
H_{u}'=\{0\}$ and the operator in question has a
compact self-commutator. This motivates us to
ask the following question which is closely
related to Corollaries~\ref{bdfir}
and~\ref{bst-chrup}.
   \begin{question}
Is an analytic finitely cyclic operator $T$ of
class $\mathcal A_{\infty}$ unitarily equivalent
to a compact perturbation of $U^n_+$, where
$n=\dim\ker(T^*)$?
   \end{question}
   \section{\label{Sec.7}Characterizations of classes $\mathcal A_k$}
   In this section, we characterize the
left-invertible operators $T$ satisfying the
condition (iii$_{k}$) of Definition~\ref{duffa},
where $k\in \zbb_+ \cup \{\infty\}$. In
particular, we obtain characterizations of
classes $\mathcal A_k.$ We begin with an
observation of independent interest.
   \begin{lemma}
Let $T \in \mathcal B(\hh)$ be left-invertible
and let $k$ be a positive integer. Then the
range of $T'^k$ is given by
   \begin{align*}
\ob{T^{\prime k}} = \bigcap_{j=0}^{k-1} \ker((I
- T'T^*)T^{*j}).
   \end{align*}
   \end{lemma}
   \begin{proof}
Using \cite[Lemma 2.1]{S01}, the fact that the
range of $T^{\prime k}$ is closed and finally
Proposition~\ref{polur}(iii), we get
   \allowdisplaybreaks
   \begin{align*}
\ob{T^{\prime k}} & = \bigg(\bigvee_{j=0}^{k-1}
T^j(\ker T^*)\bigg)^{\perp}
   \\
& = \bigcap_{j=0}^{k-1} \big(T^j(\ker
T^*)\big)^{\perp}
   \\
& = \bigcap_{j=0}^{k-1}
\big(\ob{T^j(I-P_{\ob{T}})}\big)^{\perp}
   \\
& = \bigcap_{j=0}^{k-1} \ker((I - T'T^*)T^{*j}).
   \end{align*}
This completes the proof.
   \end{proof}
The hyponormality of $T'|_{\ob{T^{\prime k}}}$
can now be characterized as follows.
   \begin{proposition} \label{ajaj-1}
Let $T \in \mathcal B(\hh)$ be left-invertible
and $k\in \zbb_+$. Set $C_k = T^{\prime *k}
T^{\prime k}$. Then $C_k$ is invertible and the
following conditions are equivalent{\em :}
   \begin{enumerate}
   \item[(i)] $T'|_{\ob{T^{\prime k}}}$ is
hyponormal,
   \item[(ii)] $C_k^{1/2} T' C_k^{-1/2}$ is hyponormal.
   \end{enumerate}
   \end{proposition}
   \begin{proof}
Since $T^{\prime k}$ is left-invertible, $C_k$ is invertible
and, by Proposition~\ref{polur}(iii), we have
   \begin{align} \label{su-1}
P_{\ob{T^{\prime k}}} = T^{\prime k}(T^{\prime
*k}T^{\prime k})^{-1}T^{\prime *k} = T^{\prime
k}C_k^{-1}T^{\prime *k}.
   \end{align}
It is easily seen that $T'|_{\ob{T^{\prime k}}}$
is hyponormal if and only if
   \begin{align*}
T^{\prime *k} T^{\prime} P_{\ob{T^{\prime
k}}}T^{\prime *} T^{\prime k} \Le T^{\prime
*(k+1)} T^{\prime (k+1)},
   \end{align*}
or equivalently, by \eqref{su-1}, if and only if
   \begin{align*}
C_k T' C_k^{-1} T^{\prime *} C_k \Le T^{\prime
*} C_k T^{\prime},
   \end{align*}
which in turn is equivalent to
   \begin{align*}
C_k^{1/2} T' C_k^{-1} T^{\prime *} C_k^{1/2} \Le
C_k^{-1/2} T^{\prime *} C_k T^{\prime}
C_k^{-1/2}.
   \end{align*}
As a consequence, $T'|_{\ob{T^{\prime k}}}$ is
hyponormal if and only if
   \begin{align*}
\|(C_k^{1/2} T^{\prime} C_k^{-1/2})^* h\|^2 \Le
\|(C_k^{1/2} T^{\prime} C_k^{-1/2}) h\|^2, \quad
h \in \hh,
   \end{align*}
which completes the proof.
   \end{proof}
   \begin{corollary} \label{sucud}
Suppose that $T\in \mathcal B(\hh)$ is left-invertible and
$T=U|T|$ is the polar decomposition of $T.$ Then
$T'|_{\ob{T'}}$ is hyponormal if and only if $|T|^{-1} U$ is
hyponormal. In particular, if $T'$ is hyponormal, then
$|T|^{-1} U$ is hyponormal.
   \end{corollary}
   \begin{proof}
It follows from Proposition~\ref{polur}(i) that
   \begin{align*}
C_1^{1/2} T' C_1^{-1/2} = |T|^{-1} T |T|^{-1} =
|T|^{-1} U.
   \end{align*}
Combined with Proposition~\ref{ajaj-1}, this
completes the proof.
   \end{proof}
   \begin{corollary}
If $T\in \mathcal B(\hh)$ is of class $\mathcal
A_k$ for some $k\in \zbb_+$, then $T'$ is
similar to a hyponormal operator.
   \end{corollary}
   \begin{corollary} \label{aglus}
Let $T \in \mathcal B(\mathcal H)$ be an
operator of class $\mathcal A_k$ for some $k\in
\zbb_+.$ Then the operator $S_k:=C^{1/2}_k T'
C^{-1/2}_k$ is a $2$-hypercontraction, that is
   \begin{align} \label{agli}
\sum_{j=0}^m (-1)^j \binom{m}{j} S_k^{*j}S_k^j
\Ge 0, \quad m=1,2.
   \end{align}
In particular,
   \begin{align} \label{posiu}
T^{\prime *k}(I - 2 T^{\prime *}T' + T^{\prime
* 2}T'^2)T^{\prime k} \Ge 0.
   \end{align}
   \end{corollary}
   \begin{proof}
By Proposition 7.2, the operator $S_k$ is
hyponormal. It follows that
   \begin{align*}
I - 2S_k^*S_k + S_k^{*2}S_k^2 \Ge I - 2S_k^*S_k
+ S_k^{*}S_kS_k^*S_k = (I-S_k^*S_k)^2 \Ge 0,
   \end{align*}
which gives \eqref{agli} for $m=2$ (this is true
for all hyponormal operators). It is now easy to
see that \eqref{agli} for $m=2$ is equivalent to
   \begin{align*}
C_k - 2T^{\prime *}C_kT^{\prime} + T^{\prime
*2}C_kT^{\prime 2} \Ge 0,
   \end{align*}
which in turn is equivalent to \eqref{posiu}.

Since $T^*T \Ge I, $ we deduce that
   \begin{align*}
C_1=T^{\prime *} T^{\prime} = (T^*T)^{-1} \Le I.
   \end{align*}
Consequently, we have
   \begin{align*}
C_{k+1} = T^{\prime * k} C_1 T^{\prime k} \Le
T^{\prime * k} T^{\prime k} = C_k, \quad k \in
\zbb_+,
   \end{align*}
or equivalently that
   \begin{align*}
C_k^{-1/2}C_{k+1} C_k^{-1/2} \Le I, \quad k\in
\zbb_+.
   \end{align*}
This gives
   \begin{align*}
S_k^*S_k = C^{-1/2}_k T^{\prime *} C_k T'
C^{-1/2}_k = C^{-1/2}_k C_{k+1} C^{-1/2}_k \Le
I, \quad k\in \zbb_+,
   \end{align*}
which shows that \eqref{agli} holds for $m=1.$
This completes the proof.
   \end{proof}
   \begin{remark}
Regarding the operator $C_k^{1/2} T' C_k^{-1/2}$ that appears
in Proposition~\ref{ajaj-1}, observe that if $A\in \mathcal
B(\hh)$ is positive and invertible and $T\in \mathcal
B(\hh)$, then the adjoint of $T$ with respect to the
(equivalent) inner product $\inp{A(\cdot)}{\mbox{-}}$, is
equal to $(ATA^{-1})^*$. In turn, if $T\in \mathcal B(\hh)$
is left-invertible and $T=U|T|$ is the polar decomposition of
$T$, then by Proposition~\ref{polur}, $T'= U|T|^{-1}$ is the
polar decomposition of $T',$ so the operator $|T|^{-1} U$
that appears in Corollary~\ref{sucud} is the Duggal
transform\footnote{If $R=W|R|$ is the polar decomposition of
$R \in \mathcal B (\hh),$ then $|R|W$ is called the {\em
Duggal transform} of $R$ (see \cite{F-J-K-P03}).} of $T'.$
Finally, the concept of $m$-hypercontractivity that appears
in Corollary~\ref{aglus} was introduced by Agler
in~\cite{Ag85}.
   \hfill $\diamondsuit$
   \end{remark}
The following characterization of the
hyponormality of $T'|_{\hh_{u}'}$ can be viewed
as an asymptotic version of
Proposition~\ref{ajaj-1} (see
Remark~\ref{tuzaref}).
   \begin{proposition} \label{inft-1}
Let $T \in \mathcal B(\hh)$ be left-invertible
and let $C_k = T^{\prime *k} T^{\prime k}$ for
$k\in \zbb_+$. Then $\{T^{\prime (k+1)} C_k^{-1}
T^{\prime * (k+1)}\}_{k=0}^{\infty}$ and
$\{T^{\prime k}C_k^{-1} C_{k+1} C_k^{-1}
T^{\prime * k}\}_{k=0}^{\infty}$ converge in
{\em SOT} $($the strong operator topology$)$ to
positive operators denoted by $A$ and $B$,
respectively, and the following conditions are
equivalent{\em :}
   \begin{enumerate}
   \item[(i)] $T'|_{\hh_{u}'}$ is
hyponormal,
   \item[(ii)] $A \Le B$,
   \item[(iii)] for all $h\in \hh$,
   \begin{align*}
\lim_{k\to \infty} \|(C_k^{1/2}T'C_k^{-1/2})^*
C_k^{-1/2} T^{\prime * k}h\|^2 \Le \lim_{k\to
\infty} \|C_k^{1/2}T'C_k^{-1/2} C_k^{-1/2}
T^{\prime * k}h\|^2,
   \end{align*}
   \item[(iv)]
$P_{\hh_{u}'} \Big((T^*T)^{-1} - T^{\prime}
P_{\hh_{u}'} T^{\prime *}\Big) P_{\hh_{u}'} \Ge
0$.
   \end{enumerate}
Moreover, $A=P_{\hh_{u}'} T^{\prime}
P_{\hh_{u}'} T^{\prime *} P_{\hh_{u}'}$ and
$B=P_{\hh_{u}'} (T^*T)^{-1} P_{\hh_{u}'}.$
   \end{proposition}
   \begin{proof}
Set $P_k=P_{\ob{T^{\prime k}}}$ for $k \in
\zbb_+$ and $P_{\infty} = P_{\hh_{u}'}$. Since
$\ob{T^{\prime k}} \searrow \hh_{u}'$ as
$k\nearrow \infty$, we deduce that $P_k \searrow
P_{\infty}$ as $k\nearrow \infty$ in SOT. By the
sequential SOT-continuity of multiplication in
$\mathcal B(\hh)$, we deduce that
   \begin{align} \label{zbiezn}
   \begin{minipage}{64ex}
{\em the sequences $\{P_k T^{\prime *}
P_k\}_{k=0}^{\infty}$ and $\{(P_k T^{\prime *}
P_k)^*P_k T^{\prime *} P_k\}_{k=0}^{\infty}$
converge in SOT to $P_{\infty} T^{\prime *}
P_{\infty}$ and $P_{\infty} T^{\prime}
P_{\infty} T^{\prime *} P_{\infty}$,
respectively,}
   \end{minipage}
   \end{align}
and by Proposition~\ref{polur}(i),
   \begin{align} \label{zbiezm}
   \begin{minipage}{64ex}
{\em the sequences $\{T' P_k\}_{k=0}^{\infty}$
and $\{(T'P_k)^*T'P_k\}_{k=0}^{\infty}$ converge
in SOT to $T'P_{\infty}$ and $P_{\infty}
T^{\prime *} T^{\prime} P_{\infty} = P_{\infty}
(T^*T)^{-1} P_{\infty}$, respectively.}
   \end{minipage}
   \end{align}
Observe now that,
   \begin{align*}
P_k T^{\prime *} P_k \overset{\eqref{su-1}}= T^{\prime
k}C_k^{-1}T^{\prime *k} T^{\prime *} T^{\prime
k}C_k^{-1}T^{\prime *k} = T^{\prime k}C_k^{-1}T^{\prime
*(k+1)}, \quad k\in \zbb_+,
   \end{align*}
which yields
   \begin{align} \notag
(P_k T^{\prime *} P_k)^*P_k T^{\prime *} P_k &= T^{\prime
(k+1)}C_k^{-1} T^{\prime * k} T^{\prime k}C_k^{-1}T^{\prime
*(k+1)}
   \\ \label{wiki-1}
& =T^{\prime (k+1)}C_k^{-1} T^{\prime *(k+1)}, \quad k\in
\zbb_+.
   \end{align}
Combined with \eqref{zbiezn}, this implies that the sequence
$\{T^{\prime (k+1)}C_k^{-1} T^{\prime
*(k+1)}\}_{n=0}^{\infty}$ converges in SOT to $A=P_{\infty}
T^{\prime} P_{\infty} T^{\prime *} P_{\infty}$, and
   \begin{align} \label{sumir}
\langle{Ah},h \rangle= \lim_{k\to \infty}
\|(C_k^{1/2}T'C_k^{-1/2})^* C_k^{-1/2} T^{\prime
* k}h\|^2, \quad h\in \hh.
   \end{align}
A similar argument shows that
   \begin{align}  \label{wiki-2}
(T'P_k)^*T'P_k = T^{\prime k}C_k^{-1}
C_{k+1}C_k^{-1}T^{\prime *k}, \quad k\in \zbb_+,
   \end{align}
so by \eqref{zbiezm} the sequence $\{T^{\prime k}C_k^{-1}
C_{k+1} C_k^{-1} T^{\prime * k}\}_{n=0}^{\infty}$ converges
in SOT to $B=P_{\infty} (T^*T)^{-1} P_{\infty}$, and
moreover:
   \begin{align} \label{sumiv}
\langle{Bh},h \rangle= \lim_{k\to \infty}
\|(C_k^{1/2}T'C_k^{-1/2}) C_k^{-1/2} T^{\prime * k}h\|^2,
\quad h\in \hh.
   \end{align}
It is easy to see that $T'|_{\hh_{u}'}$ is
hyponormal if and only if
   \begin{align*}
\|P_{\infty} T^{\prime *} P_{\infty} h\|^2 \Le
\|T' P_{\infty} h\|^2, \quad h \in \hh,
   \end{align*}
or equivalently by \eqref{zbiezn} and
\eqref{zbiezm} if and only if
   \begin{align*}
\lim_{k\to \infty} \|P_k T^{\prime *} P_k h\|^2
\Le \lim_{k\to \infty} \|T'P_k h\|^2, \quad h\in
\hh.
   \end{align*}
Combined with \eqref{wiki-1} and \eqref{wiki-2},
this shows that (i) and (ii) are equivalent.
That (ii) and (iii) are equivalent follows from
\eqref{sumir} and \eqref{sumiv}. Finally, since
$A=P_{\infty} T^{\prime} P_{\infty} T^{\prime *}
P_{\infty}$ and $B=P_{\infty} (T^*T)^{-1}
P_{\infty}$, it is easy to see that the
conditions (ii) and (iv) are equivalent. This
completes the proof.
   \end{proof}
   \begin{remark} \label{tuzaref}
Regarding Proposition~\ref{inft-1}(iii), note
that in view of Proposition~\ref{polur}(i)
applied to $T^{\prime k}$, the operator
$T^{\prime k}C_k^{-1/2}$ is an isometry. This
implies that the range of $C_k^{-1/2} T^{\prime
* k}$ is equal to $\hh$. As a consequence,
for a fixed $k\in \zbb_+$, the inequality
   \begin{align*}
\|(C_k^{1/2}T'C_k^{-1/2})^* C_k^{-1/2} T^{\prime
* k}h\|^2 \Le
\|C_k^{1/2}T'C_k^{-1/2} C_k^{-1/2} T^{\prime *
k}h\|^2, \quad h\in \hh,
   \end{align*}
is equivalent to
   \begin{align*}
\|(C_k^{1/2}T'C_k^{-1/2})^* h\|^2 \Le
\|C_k^{1/2}T'C_k^{-1/2} h\|^2, \quad h\in \hh,
   \end{align*}
which according to Proposition~\ref{ajaj-1} is
equivalent to the hyponormality of
$T'|_{\ob{T^{\prime k}}}$. From this point of
view, Proposition~\ref{inft-1} can be seen as an
asymptotic version of Proposition~\ref{ajaj-1}.
   \hfill $\diamondsuit$
   \end{remark}
   \section{$C^*$- and $W^*$-algebra analogues of $\mathcal A_k$}
We conclude the paper by discussing the
possibility of defining membership in the
classes $\mathcal A_k$, $k\in \zbb_+ \cup
\{\infty\}$, for elements of an abstract unital
$C^*$-algebra. This question is in the spirit of
\cite{B-D77,Bun78,Sz82}.

Let $\mathcal C$ be a $C^*$-algebra with unit
$e.$ Set $\mathcal C_r := \{t\in \mathcal
C\colon \|t\|\Le r\}$ for $r\in (0,\infty).$ Fix
$t\in \mathcal C$. Denote by $C^*(t)$ the
$C^*$-algebra generated by $t$ and $e.$ We say
that $t$ is {\em hyponormal} if $t^*t - tt^*$ is
a positive element of $\mathcal C$. If $t^*t$ is
invertible, then the element $t':= t
(t^*t)^{-1}$ is called the {\em Cauchy dual} of
$t$. Clearly, if $t^*t$ is invertible, then $t$
is left-invertible (with the left-inverse
$t^{\prime *}$). The converse implication is
true, as can be seen by using the
Gelfand-Naimark representation theorem (see
\cite[Theorem~12.41]{Rud91}) and
\cite[Proposition~VIII.1.14]{Co}. Namely, the
following equivalences are valid:
   \begin{align} \label{rownow}
\text{$t^*t$ is invertible $\Leftrightarrow$
$t^*t \Ge \varepsilon e$ for some $\varepsilon >
0$ $\Leftrightarrow$ $t$ is left-invertible.}
   \end{align}
In particular, if $t$ is an {\em expansion},
that is $t^*t \Ge e$, then $t$ is
left-invertible. Moreover, if $t$ is
left-invertible, then so is $t'$ and
consequently $t^{\prime k}$ for every $k\in
\zbb_+$. As a consequence of \eqref{rownow},
$c_k := t^{\prime *k} t^{\prime k}$ is
invertible. Observe also that if $t\in \mathcal
C$ is left-invertible, then
$t=(t^{\prime})^{\prime}$ so by \eqref{rownow}
and \cite[Proposition~VIII.1.14]{Co}, the
$C^*$-algebras $C^*(t)$ and $C^*(t^{\prime})$
coincide. We now give the $C^*$-algebra analogue
of the class $\mathcal A_k$ for $k\in \zbb_+.$
   \begin{definition} \label{duffa-uv} Let
$\mathcal C$ be a unital $C^*$-algebra and let
$k\in \zbb_+.$ We say that $t\in \mathcal C$ is
{\em of class} $\mathcal A_k$ if
   \begin{enumerate}
   \item[$\bullet$] the spectral radius of
$t$ is at most $1$,
   \item[$\bullet$] $t$ is an expansion,
   \item[$\bullet$] $c_k^{1/2} t' c_k^{-1/2}$ is
hyponormal.
   \end{enumerate}
   \end{definition}
Using Remark~\ref{rmk-Jan} and
Proposition~\ref{ajaj-1}, we see that
Definitions~\ref{duffa} and \ref{duffa-uv}
coincide for $k\in \zbb_+$ if $\mathcal C =
\mathcal B(\hh)$.
   \begin{remark}
First observe that if $\mathcal C$ is a
$C^*$-algebra with unit $e$ and $t\in \mathcal
C$ is of class $\mathcal A_k$ for some $k\in
\zbb_+,$ then $C^*(t)$ is commutative if and
only if $t$ is unitary, that is, $t^*t=tt^*=e.$

Note that if $t$ is a left-invertible element of a unital
$C^*$-algebra $\mathcal C$, then $t^{\prime}=t$ if and only
if $t$ is an isometric element of $\mathcal C$ (use the fact
that $t^{\prime *}$ is a left-inverse of $t$). It is worth
mentioning that the $C^*$-algebra generated by a single
isometry was studied in \cite{Co67}.

If $\mathcal C_1$ and $\mathcal C_2$ are two unital
$C^*$-algebras, $\pi\colon \mathcal C_1 \to \mathcal C_2$ is
a $*$-isomorphism and $t_1\in \mathcal C_1$ and $t_2\in
\mathcal C_2$ are such that $\pi(t_1) = t_2,$ then for each
$k\in \zbb_+$, $t_1$ is of class $\mathcal A_k$ if and only
if $t_2$ is of class $\mathcal A_k.$
   \hfill $\diamondsuit$
   \end{remark}
The case of the class $\mathcal A_{\infty}$ is
more delicate. As we will show below, the
membership in this class can be satisfactorily
defined in the context of $W^*$-algebras, which
are more special objects compared to
$C^*$-algebras (see Definitions~ \ref{duffa-w1s}
and \ref{duffa-inf}). Several properties of
certain sequences associated with an expansion
are needed for this.
   \begin{proposition} \label{DOD-WL}
Let $\mathcal C$ be a $C^*$-algebra with unit
$e$ and $t\in \mathcal C$ be an expansion.
Define the sequences $\{p_k\}_{k=0}^{\infty},$
$\{u_k\}_{k=0}^{\infty}$ and
$\{v_k\}_{k=0}^{\infty}$ in $\mathcal C$ by
   \begin{align} \label{dufen}
\text{$p_k=t^{\prime k}c_k^{-1}t^{\prime *k},$
$u_k= t^{\prime (k+1)} c_k^{-1} t^{\prime *
(k+1)}$ and $v_k = t^{\prime k}c_k^{-1} c_{k+1}
c_k^{-1} t^{\prime * k}.$}
   \end{align}
Then the following statements hold for all $k\in
\zbb_+${\em :}
   \begin{enumerate}
   \item[(i)]
$p_k=p_k^*p_k,$ $p_kp_{k+1}=p_{k+1}$ and $0 \Le
p_{k+1} \Le p_k \Le e,$
   \item[(ii)]
$0 \Le c_{k+1} \Le c_k \Le e$ and $c_k^{-1/2}
c_{k+1} c_k^{-1/2} \Le e,$
   \item[(iii)]
$u_k=t^{\prime} p_k t^{\prime *},$ $0 \Le
t^{\prime (k+1)} t^{\prime * (k+1)} \Le u_k \Le
p_{k+1} \Le e$ and $u_{k+1} \Le u_k,$
   \item[(iv)]
$v_{k+1} = p_{k+1} v_k p_{k+1},$ $v_k = p_k
t^{\prime *}t^{\prime} p_k$ and $0 \Le v_k \Le
p_k \Le e.$
   \end{enumerate}
Moreover, $t$ is an expansive element of $C^*(t)$ and the
elements $t^{\prime}$, $c_k$, $c_k^{-1}$, $p_k$, $u_k$ and
$v_k$ belong to $C^*(t).$
   \end{proposition}
   \begin{proof}
(i) By definition, $p_k=p_k^*$ and
   \begin{align} \label{pikay}
p_k^2 = t^{\prime k}c_k^{-1}(t^{\prime *k}
t^{\prime k})c_k^{-1}t^{\prime *k} =p_k, \quad
k\in \zbb_+,
   \end{align}
so $p_k=p_k^*p_k$ for all $k\in \zbb_+.$ Since
   \begin{align} \label{pikaypi}
p_kp_{k+1} = t^{\prime k}c_k^{-1}(t^{\prime *k}
t^{\prime k})t' c_{k+1}^{-1}t^{\prime *(k+1)}=
p_{k+1}, \quad k\in \zbb_+,
   \end{align}
and thus $p_kp_{k+1}=p_{k+1}p_k$ for all $k\in
\zbb_+,$ we deduce that
   \begin{align*}
(p_k-p_{k+1})^2 \overset{\eqref{pikay}}= p_k - 2
p_k p_{k+1} + p_{k+1} \overset{\eqref{pikaypi}}=
p_k-p_{k+1}, \quad k \in \zbb_+.
   \end{align*}
Hence, $p_{k+1} \Le p_k$ for all $k\in \zbb_+$.
Since $e-p_k$ is a selfadjoint idempotent, we
see that $p_k \Le e$ for all $k\in \zbb_+.$ This
proves (i).

(ii) We can argue as in the second paragraph of
the proof of Corollary~\ref{aglus} (the
arguments given there are purely
$C^*$-algebraic).

(iii) The identity $u_k=t^{\prime} p_k t^{\prime *}$ is
obvious. It follows from (ii) that $c_k^{-1} \Ge e$ and
$c_k^{-1} \Le c_{k+1}^{-1}$ for all $k\in \zbb_+$, so
   \begin{align*}
0 \Le t^{\prime (k+1)} t^{\prime * (k+1)} & \Le
\overset{u_k}{\overbrace{t^{\prime (k+1)}
c_k^{-1} t^{\prime
* (k+1)}}}
   \\
& \Le t^{\prime (k+1)} c_{k+1}^{-1} t^{\prime *
(k+1)} = p_{k+1} \overset{\mathrm{(i)}} \Le e,
\quad k\in \zbb_+,
   \end{align*}
which gives the second part of (iii). The third
part of (iii) follows from
   \begin{align*}
u_{k+1} = t^{\prime} p_{k+1} t^{\prime *}
\overset{\mathrm{(i)}} \Le t^{\prime} p_k
t^{\prime *} = u_k, \quad k \in \zbb_+.
   \end{align*}

(iv) It is easy to see that the identity $v_k =
p_k t^{\prime *}t^{\prime} p_k$ is valid. Hence,
by (i) and (ii), we have
   \begin{align*}
0 \Le v_k = p_k c_1 p_k \Le p_k^2=p_k \Le e,
\quad k \in \zbb_+,
   \end{align*}
which proves the third part of (iv). Finally, we
come to the conclusion that
   \begin{align*}
v_{k+1} = p_{k+1} (t^{\prime *} t^{\prime})
p_{k+1} \overset{\mathrm{(i)}}= p_{k+1}
(p_k(t^{\prime *} t^{\prime})p_k) p_{k+1} =
p_{k+1} v_k p_{k+1}, \quad k\in \zbb_+,
   \end{align*}
which gives the first identity in (iv).

The proof of the ``moreover'' part is as
follows. That $t$ is an expansive element of
$C^*(t)$ and $t^{\prime}, c_k, c_k^{-1} \in
C^*(t)$ for every $k\in \zbb_+$ follows from
\eqref{rownow} and the fact that
$\sigma_{C^*(t)}(s)=\sigma_{\mathcal{C}}(s)$ for
every $s \in C^*(t)$ (see
\cite[Proposition~VIII.1.14]{Co}; see also
\cite[Theorem~VIII.3.6]{Co}). As a consequence,
each of the elements $p_k$, $u_k$ and $v_k$
belongs to $C^*(t).$ This completes the proof.
   \end{proof}
The following property of elements of class
$\mathcal A_k$ is a direct consequence of
Proposition~\ref{DOD-WL} and
\cite[Proposition~VIII.1.14]{Co}. We will call
it $C^*$-permanence property which is shared by
some other classes of elements of a unital
$C^*$-algebra, including hyponormal ones.
   \begin{corollary}  \label{permi}
If $\mathcal C$ is a unital $C^*$-algebra, $t\in
\mathcal C$ and $k\in \zbb_+,$ then $t$ is of
class $\mathcal A_k$ if and only if $t$ is of
class $\mathcal A_k$ as an element of $C^*(t).$
   \end{corollary}
The question of the existence of limits of sequences that
were defined in \eqref{dufen} is discussed below in the
context of $W^*$-algebras. Given an element $t$ of a
$W^*$-algebra $\mathcal W$, we denote by $W^*(t)$ the
$W^*$-algebra generated by $t$ and the unit element of
$\mathcal W,$ that is, $W^*(t)$ is the $\sigma(\mathcal W,
\mathcal W_{*})$-closure of $C^*(t)$, where $\mathcal W_{*}$
is the predual of $\mathcal W$ (see
\cite[Corollary~1.7.9]{Sak98}). Recall that every
$W^*$-algebra always has a unit (see \cite[Sec.\
1.7]{Sak98}).
   \begin{proposition} \label{wstaralg}
Let $\mathcal W$ be a $W^*$-algebra with predual
$\mathcal W_{*}$ and $t\in \mathcal W$ be an
expansion. Let $\{p_k\}_{k=0}^{\infty},$
$\{u_k\}_{k=0}^{\infty}$ and
$\{v_k\}_{k=0}^{\infty}$ be as in {\em
\eqref{dufen}}. Then
   \begin{enumerate}
   \item[(i)] the sequence $\{p_k\}_{k=0}^{\infty}$
is $\sigma(\mathcal W, \mathcal
W_{*})$-convergent to a projection $($i.e., a
selfadjoint idempotent$)$ of $\mathcal W$
denoted by $p_{\infty},$
   \item[(ii)] $p_{\infty}$ is the infimum of
$\{p_k\}_{k=0}^{\infty}$ computed in the set of
all selfadjoint elements of $\mathcal W$
$($equivalently, in the set of all projections
of $\mathcal W$$),$
   \item[(iii)] the sequences $\{u_k\}_{k=0}^{\infty}$
and $\{v_k\}_{k=0}^{\infty}$ are
$\sigma(\mathcal W, \mathcal W_{*})$-convergent
to positive elements of $\mathcal W$ denoted by
$a$ and $b$, respectively.
   \end{enumerate}
Moreover, $t$ is an expansive element of $W^*(t)$ and the
elements $t^{\prime}$, $c_k$, $c_k^{-1}$, $p_k$, $u_k,$
$v_k,$ $p_{\infty},$ $a$ and $b$ belong to $W^*(t).$
   \end{proposition}
   \begin{proof}
(i)\&(ii) By Proposition~\ref{DOD-WL}(i),
$\{p_k\}_{k=0}^{\infty}$ is a uniformly bounded
monotonically decreasing sequence of projections
of $\mathcal W.$ It can be deduced from
\cite[Lemma~1.7.4, Theorem~1.16.7 and
Corollary~1.15.6]{Sak98} and
\cite[Theorem~4.3.5]{Ml91} that the sequence
$\{p_k\}_{k=0}^{\infty}$ is $\sigma(\mathcal W,
\mathcal W_{*})$-convergent to a projection of
$\mathcal W$, call it $p_{\infty}$, which is
equal to the infimum of $\{p_k\}_{k=0}^{\infty}$
computed in the set of all selfadjoint elements
of $\mathcal W$. Consequently, $p_{\infty}$
coincides with the infimum of
$\{p_k\}_{k=0}^{\infty}$ computed in the set of
all projections of $\mathcal W$.

(iii) It follows from
\cite[Theorem~1.16.7]{Sak98} that there exists a
faithful $W^*$-representation $\pi\colon
\mathcal W \to \mathcal B(\hh)$ of $\mathcal W.$
By \cite[Proposition~1.16.2]{Sak98},
$\pi(\mathcal W)$ is $\sigma(\mathcal
B(\hh),\mathcal B(\hh)_{*})$-closed. In turn,
according to \cite[Proposition~1.15.1]{Sak98},
$\pi(\mathcal W)$ is SOT-closed. Since $\pi$ is
an isometric $W^*$-homomorphism, we deduce from
\cite[Corollary~1.15.6]{Sh02} that for every
$r\in (0,\infty)$, the mapping
   \begin{align*}
\pi_r:=\pi|_{\mathcal W_r} \colon \mathcal W_r
\to \pi(\mathcal W)_r
   \end{align*}
is continuous if $\mathcal W_r$ and
$\pi(\mathcal W)_r$ are equipped with the
topologies $\sigma(\mathcal W, \mathcal W_{*})$
and $\sigma$-WOT (the $\sigma$-weak operator
topology), respectively. Because $\pi_r$ is a
bijection, we infer from the Banach-Alaoglu
theorem that $\pi_r$ is a homeomorphism. By
setting $T=\pi(t)$, it can be verified that
   \begin{align} \label{takufir}
\pi(v_k) = T^{\prime k}C_k^{-1} C_{k+1} C_k^{-1}
T^{\prime * k}, \quad k\in \zbb_+.
   \end{align}
It follows from Proposition~\ref{inft-1} that
$\{\pi(v_k)\}_{k=0}^{\infty}$ converges in SOT, and
consequently in $\sigma$-WOT to a positive operator $B\in
\mathcal B(\hh).$ Since $\pi(\mathcal W)$ is SOT-closed,
there exists a (unique) positive element $b\in \mathcal W$
such that $B=\pi(b).$ By the uniform boundedness principle,
there exists $r\in (0,\infty)$ such that $\pi(b), \pi(v_k)
\in \pi(\mathcal W)_r$ for all $k\in \zbb_+.$ Since $\pi_r$
is a homeomorphism, the sequence $\{v_k\}_{k=0}^{\infty}$
converges in $\sigma(\mathcal W, \mathcal W_{*})$ to $b.$ The
same reasoning can be applied to the sequence
$\{u_k\}_{k=0}^{\infty}.$

The ``moreover'' part is a direct consequence of the
statements (i) and (iii), the ``moreover'' part of
Proposition~\ref{DOD-WL} and \cite[Corollary~1.7.9]{Sak98}.
This completes the proof.
   \end{proof}
After this preparation, we can give two
equivalent definitions of the class $\mathcal
A_{\infty}$ in the context of $W^*$-algebras.
The first one is patterned on the equivalence
(i)$\Leftrightarrow$(iv) of
Proposition~\ref{inft-1}.
   \begin{definition} \label{duffa-w1s}
Let $\mathcal W$ be a $W^*$-algebra. We say that
$t\in \mathcal W$ is {\em of class}~$\mathcal
A_{\infty}$~if
   \begin{enumerate}
   \item[(i)] the spectral radius of
$t$ is at most $1$,
   \item[(ii)] $t$ is an expansion,
   \item[(iii)] $p_{\infty} \big((t^*t)^{-1} - t^{\prime}
p_{\infty} t^{\prime *}\big) p_{\infty} \Ge 0,$
   \end{enumerate}
where $p_{\infty}$ is as in
Proposition~\ref{wstaralg}.
   \end{definition}
The second definition is based on the
equivalence (i)$\Leftrightarrow$(ii) of
Proposition~\ref{inft-1}.
   \begin{definition} \label{duffa-inf}
Let $\mathcal W$ be a $W^*$-algebra. We say that
$t\in \mathcal W$ is {\em of class}~$\mathcal
A_{\infty}$~if
   \begin{enumerate}
   \item[(i)] the spectral radius of
$t$ is at most $1$,
   \item[(ii)] $t$ is an expansion,
   \item[(iii)] $a\Le b,$
   \end{enumerate}
where $a$ and $b$ are as in
Proposition~\ref{wstaralg}.
   \end{definition}
   \begin{theorem} \label{sarown}
Definitions~{\em \ref{duffa-w1s}} and {\em
\ref{duffa-inf}} are equivalent.
   \end{theorem}
   \begin{proof}
Without loss of generality, we can assume that
$t$ is an expansive element of $\mathcal W.$ Let
$\pi\colon \mathcal W \to \mathcal B(\hh)$ be a
faithful $W^*$-representation of $\mathcal W$
(see \cite[Theorem~1.16.7]{Sak98}). Set
$T=\pi(t)$ and $\hh_u'=\bigcap_{n=0}^{\infty}
\ob{T^{\prime n}}.$ By
Proposition~\ref{wstaralg} and
\cite[Corollary~1.15.6]{Sak98},
$\{\pi(p_k)\}_{k=0}^{\infty}$ is
$\sigma$-WOT-convergent to $\pi(p_{\infty}).$
Since by \eqref{su-1}, $\pi(p_k) =
P_{\ob{T^{\prime k}}}$ for every $k \in \zbb_+,$
we see that $\pi(p_k) \searrow P_{\hh_{u}'}$ as
$k\nearrow \infty$ in SOT (see the proof of
Proposition~\ref{inft-1}). As a consequence,
   \begin{align} \label{pipi}
\pi(p_{\infty}) = P_{\hh_{u}'}.
   \end{align}
According to the proof of
Proposition~\ref{wstaralg}(iii),
$\{\pi(v_k)\}_{k=0}^{\infty}$ is SOT-convergent
to $\pi(b).$ By \eqref{takufir}, the
``moreover'' part of Proposition~\ref{inft-1}
and \eqref{pipi}, we have
   \begin{align} \label{pipu}
\pi(b) = B = P_{\hh_{u}'} (T^*T)^{-1}
P_{\hh_{u}'} = \pi(p_{\infty} (t^*t)^{-1}
p_{\infty}).
   \end{align}
The same reasoning applied to the sequence
$\{u_k\}_{k=0}^{\infty}$ leads to the identity
   \begin{align} \label{pupu}
\pi(a) = A=P_{\hh_{u}'} T^{\prime} P_{\hh_{u}'}
T^{\prime *} P_{\hh_{u}'} = \pi(p_{\infty}
t^{\prime} p_{\infty} t^{\prime *} p_{\infty}).
   \end{align}
Clearly, the condition (iii) of
Definition~\ref{duffa-inf} is equivalent to
$\pi(b-a) \Ge 0,$ which by \eqref{pipu} and
\eqref{pupu} is equivalent to the condition
(iii) of Definition~\ref{duffa-w1s}. This
completes the proof.
   \end{proof}
The following is the counterpart of Corollary~\ref{permi} for
elements of class $\mathcal A_{\infty}.$
   \begin{corollary}
If $\mathcal W$ is a $W^*$-algebra and $t\in \mathcal W,$
then $t$ is of class $\mathcal A_{\infty}$ if and only if $t$
is of class $\mathcal A_{\infty}$ as an element of $W^*(t).$
   \end{corollary}
   \begin{proof}
It follows from Proposition~\ref{wstaralg} and
\cite[Proposition~VIII.1.14]{Co} that the infimum of
$\{p_k\}_{k=0}^{\infty}$ computed in the set of all
projections of $\mathcal W$ coincides with the infimum of
$\{p_k\}_{k=0}^{\infty}$ computed in the set of all
projections of $W^*(t).$ Hence, by Theorem~\ref{sarown}, we
can apply Definition~\ref{duffa-w1s} and
\cite[Proposition~VIII.1.14]{Co} again.
   \end{proof}
One can see by applying Remark~\ref{rmk-Jan},
Propositions~\ref{inft-1} and \ref{wstaralg} and
\cite[Corollary~1.15.6]{Sak98} that Definitions~\ref{duffa}
and \ref{duffa-inf} coincide for $k=\infty$ if $\mathcal W =
\mathcal B(\hh)$. Regarding Proposition~\ref{wstaralg}(ii),
it is worth pointing out that the property of being closed
under the operation of taking infimum is one of the two basic
properties characterizing $W^*$-algebras among $C^*$-algebras
(cf.\ \cite[Definition~2.1]{Kad85}).
   {}
\end{document}